\documentclass[a4paper,11pt,reqno,dvipsnames]{amsart}
\usepackage{marginnote}
\usepackage{amsmath}
\usepackage{mathtools}
\usepackage[margin=1.45in]{geometry}
\usepackage{amsfonts}
\usepackage{amscd}
\usepackage{amssymb}
\usepackage{graphicx}
\usepackage[justification=centering]{caption}
\usepackage{mathrsfs}
\usepackage{dsfont}
\usepackage{tikz}

\makeatletter
\g@addto@macro\bfseries{\boldmath}
\makeatother

\usepackage[dvips,all,arc,curve,color,frame]{xy}
\usepackage{enumitem}

\usepackage[colorlinks]{hyperref}
\usepackage[nameinlink]{cleveref}
\usepackage{tikz,mathrsfs}
\usepackage{pgfplots}
\pgfplotsset{compat=1.15}
\usetikzlibrary{arrows,decorations.pathmorphing,decorations.pathreplacing,positioning,shapes.geometric,shapes.misc,decorations.markings,decorations.fractals,calc,patterns}

\newcommand{\f}{\frac}

\newtheorem{thm}{Theorem}[section]
\newtheorem{prop}[thm]{Proposition}
\newtheorem{lemma}[thm]{Lemma}
\newtheorem{claim}[thm]{Claim}

\theoremstyle{definition}

\newtheorem{conj}[thm]{Conjecture}
\newtheorem{question}[thm]{Question}

\begin{document}
\author{Harry Metrebian}
\address{Department of Pure Mathematics and Mathematical Statistics, University of Cambridge, UK}
\email{rhkbm2@cam.ac.uk}

\title[Odd colouring on the torus]{Odd colouring on the torus}

\begin{abstract} 
A proper vertex-colouring of a simple graph $G$ is said to be odd if, for every non-isolated vertex $v$ of $G$, some colour appears an odd number of times in the neighbourhood of $v$. We show that if $G$ embeds in the torus, then it admits a proper odd vertex-colouring with at most $9$ colours.
\end{abstract}

\maketitle

\section{Introduction}

The concept of odd colouring was recently introduced by Petru\v{s}evski and \v{S}krekovski \cite{ps21}. A \emph{proper odd vertex-colouring} of a simple graph $G$, often referred to as an \emph{odd colouring} for short, is a proper colouring $c$ of $V(G)$ with the property that, for every non-isolated vertex $v$ of $G$, there exists a colour $i$ such that $|c^{-1}(i) \cap N(v)|$ is odd: in other words, for every $v$, there is some colour that appears an odd number of times in the neighbourhood of $v$. The \emph{odd chromatic number} of a graph $G$, denoted $\chi_o(G)$, is the smallest number $k$ such that $G$ admits an odd colouring using $k$ colours.

Petru\v{s}evski and \v{S}krekovski \cite{ps21} used the discharging method to prove that every planar graph $G$ satisfies $\chi_o(G) \leq 9$, and conjectured that in fact $\chi_o(G) \leq 5$ for all planar graphs. Note that if true, this conjecture is the best possible, since $C_5$ requires $5$ colours. Petr and Portier \cite{pp22} improved the upper bound to $\chi_o(G) \leq 8$, and Fabrici et al.\ \cite{fabrici22} showed independently that in fact every planar graph has a \emph{conflict-free colouring} with at most $8$ colours, meaning that there is a colouring such that for every $v$ some colour appears exactly once in its neighbourhood. Clearly any conflict-free colouring is an odd colouring.

A \emph{toroidal graph} is a simple graph that can be embedded on the surface of a torus. We will demonstrate an upper bound on $\chi_o(G)$ for toroidal graphs:

\begin{thm}\label{mainthm}
Let $G$ be a toroidal graph. Then $\chi_o(G) \leq 9$.
\end{thm}

Note that at least $7$ colours are sometimes required, since $K_7$ can be embedded in the torus.

Our proof mainly builds on the discharging techniques of Petru\v{s}evski and \v{S}krekovski \cite{ps21}. Our application of the discharging method is rather more sensitive and leaves a special case that must be dealt with separately.

\section{Proof outline and preliminaries}

For convenience, we will refer to a proper odd colouring with at most $9$ colours as a \emph{nice colouring}. From now on, we assume that there exists a toroidal graph which does not have a nice colouring, and $G = (V,E)$ will always denote a minimal such graph. Clearly $G$ is connected.

In our proof, we will often be faced with a situation where we have a colouring $c$ of $G\setminus\{v\}$, and we would like to find a colour that we can use at $v$ to extend $c$ to a nice colouring of $G$. There are two reasons why we might not be able to use a particular colour at $v$. Firstly, for each $w \in N(v)$, we cannot use $c(w)$ at $v$, since the resulting colouring would not be proper. We refer to such colours as \emph{forbidden at $v$ by properness}. Secondly, for each $w \in N(v)$, there can be at most one colour $b(w)$ such that using $b(w)$ at $v$ would result in each colour being used an even number of times in $N(w)$. We refer to such colours as \emph{forbidden at $v$ by oddness}. There are therefore at most $2d(v)$ colours forbidden at $v$ in total. 

We will say that a colouring is \emph{proper at $v$} if $v$ is coloured differently from all its neighbours, and \emph{odd at $v$} to mean that some colour appears an odd number of times in $N(v)$. We will also use this terminology for partial colourings as long as they are defined at all the vertices of $N(v)$. Note that if $v$ has odd degree, then any colouring is automatically odd at $v$.

Note that the torus has Euler characteristic $0$, so for any toroidal graph we have a version of Euler's formula: $$|V(G)| - |E(G)| + |F(G)| \geq 0.$$ Since $3|F(G)| \leq 2|E(G)|$, we have that $|E(G)| \leq 3|V(G)|$, and hence $\delta(G) \leq 6$. In order for equality to hold, $G$ must be $6$-regular and every face of $G$ must have 3 edges. In addition, equality must hold in Euler's formula, so every face of $G$ must be homeomorphic to a disk (see  \cite{mohar}, Chapter 3). In other words, $\delta(G) = 6$ iff $G$ is a $6$-regular triangulation of the torus.

In fact, for our minimal $G$, we must have $\delta(G) \geq 5$. This was proved by Petru\v{s}evski and \v{S}krekovski \cite{ps21} in the context of planar graphs, but still holds for toroidal graphs. For completeness, we reproduce their proof here.

\begin{claim}\label{mindeg}
$G$ has minimum degree $\delta(G) \geq 5$.
\end{claim}

\begin{proof}
Suppose that $G$ contains a vertex $v$ of degree $1$ or $3$. Let $c$ be a nice colouring of $G \setminus \{v\}$, which exists by minimality. We would like to extend $c$ to a nice colouring of $G$. There are now at most $3$ colours forbidden at $v$ by properness and at most $3$ forbidden by oddness, so there is some colour left over that can be used at $v$. Since $v$ has odd degree, the resulting colouring is odd at $v$ and therefore nice, which is a contradiction.

Now suppose instead that $G$ contains a vertex $v$ of degree $2$ or $4$. Let $w$ be an arbitrary neighbour of $v$, and let $G'$ be the graph constructed from $G$ by removing $v$ and then joining $w$ to every vertex in $N_G(v)$ to which it is not already adjacent. The graph $G'$ is toroidal, and so by minimality it has a nice colouring $c$. Now we return to $G$ and colour the vertices of $G\setminus\{v\}$ according to $c$. There are at most $8$ colours forbidden at $v$, leaving a colour that can be used at $v$. The colour $c(w)$ appears exactly once in $N(v)$ by construction, so the resulting colouring is odd at $v$. We therefore have a nice colouring of $G$, which is again a contradiction.
\end{proof}

The rest of the proof will be structured as follows. We first use the discharging method and the minimality of $G$ to show that $G$ cannot contain a vertex of degree $5$. This leaves the case where $G$ is a $6$-regular triangulation. We then use the classification of $6$-regular triangulations of the torus by Altshuler \cite{altshuler} to show that every such triangulation admits a nice colouring, finishing the proof.

As is standard, we will refer to a vertex of degree $d$ as a $d$-vertex, and a vertex of degree at least $d$ as a $d^+$-vertex. The boundary of a face $f$ need not be connected, but it can be considered as a disjoint union of closed walks; we define the \emph{size} of $f$, denoted $d(f)$, to be the total number of vertices appearing in this union of walks, counting with multiplicity. For example, a face whose boundary is a $k$-cycle has size $k$. We call a face of size $k$ a $k$-face, and a face of size at least $k$ a $k^+$-face.

\section{Discharging method}\label{discharging}

In this section, we will deal with the case $\delta(G) = 5$. We will use the fact that the torus is locally homeomorphic to the plane, and therefore we will draw subgraphs of $G$ as if they were on the plane, although this may not always faithfully represent the embedding in the torus. However, the order of the edges around each vertex will be unambiguous and represented correctly.

Throughout the proof, there will be some occasions on which two vertices that are given different names could in fact be the same, or where two named vertices which are not defined to be adjacent could in fact be adjacent. However, this will never affect our arguments.

First, we will need a simple observation made by Petru\v{s}evski and \v{S}krekovski \cite{ps21}.

\begin{claim}\label{5neighbour}
Suppose $v \in G$ is a $5$-vertex. Then $v$ has at most one neighbour of odd degree.
\end{claim}

\begin{proof}
Suppose that $v$ has two neighbours $x$ and $y$ of odd degree. Let $c$ be a nice colouring of $G\setminus\{v\}$; we would like to extend $c$ to $v$. There are at most $5$ colours forbidden at $v$ by properness, and since $x$ and $y$ have odd degree, they cannot forbid colours at $v$ by oddness, so there are at most $3$ colours forbidden by oddness. This means there is a colour left over that can be used at $v$, and since $v$ has odd degree, the resulting colouring is nice, which is a contradiction.
\end{proof}

We will now introduce our discharging rules. As is standard, we begin by assigning to every vertex $v \in G$ a charge $d(v)-6$ and to every face $f$ a charge $2d(f)-6$. By Euler's formula, the total charge over all vertices and faces of the graph is at most $0$.

The rules are as follows. If a vertex appears with multiplicity greater than $1$ in the closed walk around the boundary of a face, we consider each appearance to be a separate vertex when applying the rules.
\begin{enumerate}[label={(R\arabic*)},itemsep=5pt]
    \item Every $5^+$-face sends charge $1.1$ to each incident $5$-vertex.
    \item Every $4$-face sends charge $1$ to each incident $5$-vertex, unless its incident vertices are, in order, two adjacent $5$-vertices and two adjacent $6^+$-vertices, in which case it sends charge $\f{3}{4}$ to each $5$-vertex.
    \item If $u$ and $v$ are $6^+$-vertices on a $4^+$-face $f$ that are adjacent along an edge of $f$ and also both incident to a $3$-face $uvw$, and $w$ is a $5$-vertex, then $f$ sends charge $\f{1}{2}$ to $w$.
    \item Suppose $v$ is a $7^+$-vertex with at least one neighbouring $5$-vertex. Let a \emph{block} be a maximal set of $5$-vertices in $N(v)$ that appear consecutively in order around $v$. Now $v$ distributes its charge evenly between the blocks, and within each block the charge is distributed evenly between the vertices. For example, an $8$-vertex with $4$ neighbouring $5$-vertices in blocks of size $2$, $1$ and $1$ sends charge $\f{2}{3}$ to each vertex in a block of size $1$, and $\f{1}{3}$ to each of the vertices in the block of size $2$.
\end{enumerate}

\begin{figure}[ht]\centering
\begin{tikzpicture}[scale=1.4]
\tikzset{enclosed/.style={draw, circle, inner sep=0pt, minimum size=1.49mm, fill=black}}
\node[enclosed, label={above: 5}] (v1) at (0,1) {};
\node[enclosed] (v2) at (0.95,0.31) {};
\node[enclosed] (v3) at (0.59,-0.81) {};
\node[enclosed] (v4) at (-0.59,-0.81) {};
\node[enclosed] (v5) at (-0.95,0.31) {};

\draw (v1) -- (v2) -- (v3) -- (v4) -- (v5) -- (v1) {};
\draw[->] (0,0) -- (0,0.8) node[midway, right] (charge) {1.1};

\node at (0,-1.8) {(R1)};

\node[enclosed, label={above: 5}] (w1) at (3.3,1) {};
\node[enclosed, label={right: $6^+$}] (w2) at (4.3,0) {};
\node[enclosed, label={below: 5}] (w3) at (3.3,-1) {};
\node[enclosed, label={left: $6^+$}] (w4) at (2.3,0) {};

\draw (w1) -- (w2) -- (w3) -- (w4) -- (w1) {};
\draw[->] (3.3,0.1) -- (3.3,0.8) node[midway, right] (charge2) {1};
\draw[->] (3.3,-0.1) -- (3.3,-0.8) node[midway, right] (charge3) {1};

\node[enclosed, label={above: 5}] (x1) at (6.5,1) {};
\node[enclosed, label={right: 5}] (x2) at (7.5,0) {};
\node[enclosed, label={below: $6^+$}] (x3) at (6.5,-1) {};
\node[enclosed, label={left: $6^+$}] (x4) at (5.5,0) {};

\draw (x1) -- (x2) -- (x3) -- (x4) -- (x1) {};
\draw[->] (6.5,0.1) -- (6.5,0.8) node[midway, right] (charge2) {$\f{3}{4}$};
\draw[->] (6.6,0) -- (7.3,0) node[midway, below] (charge3) {$\f{3}{4}$};

\node at (4.9,-1.8) {(R2)};

\node[enclosed] (y1) at (-0.25,-3) {};
\node[enclosed] (y2) at (1.25,-3) {};
\node[enclosed, label={left: $6^+$}] (y3) at (-0.25,-4.5) {};
\node[enclosed, label={right: $6^+$}] (y4) at (1.25,-4.5) {};
\node[enclosed, label={below: 5}] (y5) at (0.5,-5.8) {};
\node at (-0.05,-4.3) {$u$};
\node at (1.05,-4.3) {$v$};
\node at (0.75,-5.8) {$w$};

\draw (y1) -- (y2) -- (y4) -- (y3) -- (y1) {};
\draw (y3) -- (y5) -- (y4) {};
\draw[->] (0.5,-3.75) -- (0.5,-5.6) {};
\node[right] at (0.5,-5) {$\f{1}{2}$};

\node at (0.5,-6.8) {(R3)};

\node[enclosed] (u) at (5,-4.5) {};
\node[enclosed, label={above: 5}] (z1) at (5,-3) {};
\node[right] at (5,-3.5) (l1) {$\f{2}{3}$};
\node[enclosed, label={above right,xshift=-0.5mm,yshift=-0.5mm: $6^+$}] (z2) at (6.06,-3.44) {};
\node[enclosed, label={right: 5}] (z3) at (6.5,-4.5) {};
\node[below] at (6,-4.5) (l3) {$\f{1}{3}$};
\node[enclosed, label={below right,xshift=-0.5mm,yshift=0.5mm: 5}] (z4) at (6.06,-5.56) {};
\node[below left] at (5.71,-5.21) (l4) {$\f{1}{3}$};
\node[enclosed, label={below: $6^+$}] (z5) at (5,-6) {};
\node[enclosed, label={below left,xshift=0.5mm,yshift=0.5mm: $6^+$}] (z6) at (3.94,-5.56) {};
\node[enclosed, label={left: 5}] (z7) at (3.5,-4.5) {};
\node[below] at (4,-4.5) (l7) {$\f{2}{3}$};
\node[enclosed, label={above left,xshift=0.5mm,yshift=-0.5mm: $6^+$}] (z8) at (3.94,-3.44) {};

\begin{scope}[decoration={markings, mark=at position 0.67 with {\arrow{>}}}]
\draw[postaction={decorate}] (u) -- (z1) {};
\draw[postaction={decorate}] (u) -- (z3) {};
\draw[postaction={decorate}] (u) -- (z4) {};
\draw[postaction={decorate}] (u) -- (z7) {};
\end{scope}
\draw (u) -- (z2) {};
\draw (u) -- (z5) {};
\draw (u) -- (z6) {};
\draw (u) -- (z8) {};

\node at (5,-6.8) {(R4)};
\end{tikzpicture}
\caption{Examples of the discharging rules. The numbers next to vertices indicate their degrees, and the arrows indicate movement of charge.}\label{dischargingfig}
\end{figure}
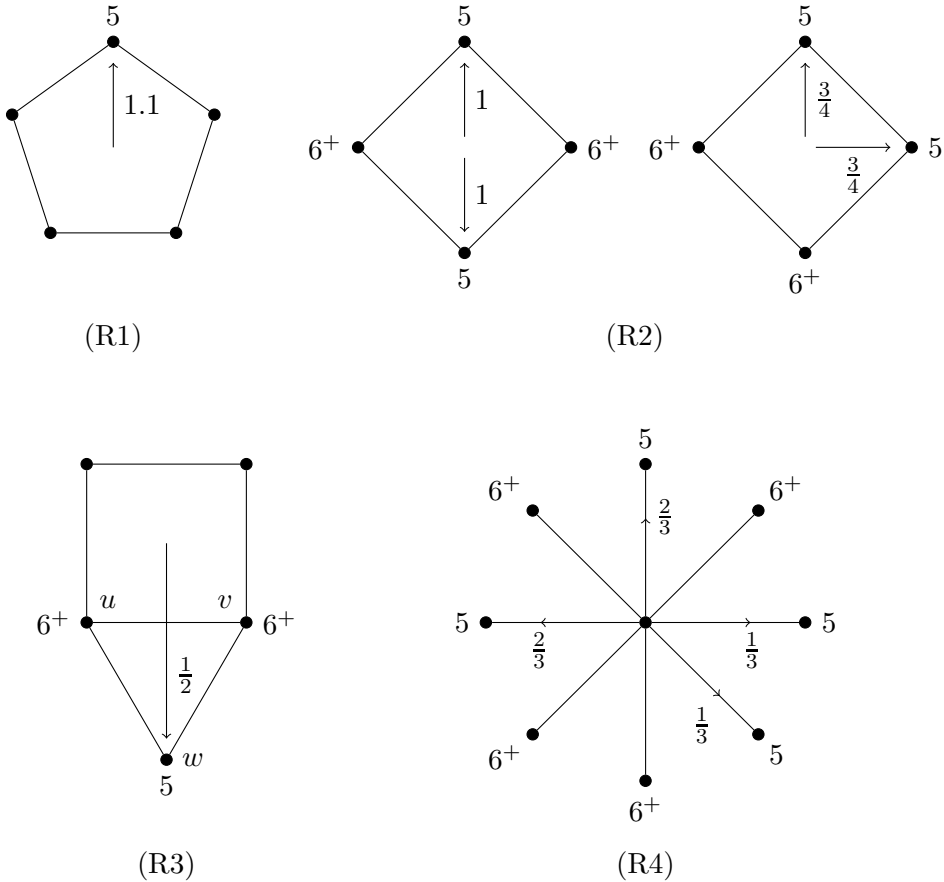

\begin{claim}\label{chargeswork}
After discharging, all faces and $6^+$-vertices have non-negative charge.
\end{claim}

\begin{proof}
The only way in which a vertex can lose charge is by (R4), and by definition no vertex can end with negative charge after the application of this rule. Therefore a vertex can only end with negative charge if it began with negative charge, and by Claim \ref{mindeg}, this is only the case for $5$-vertices.

We now turn to faces. When considering the total charge given out by a face $f$, we can imagine the charge $\f{1}{2}$ distributed by (R3) as being split into two $\f{1}{4}$ charges, one given to $u$ and one to $v$. Each $6^+$-vertex incident to $f$ is on at most two edges for which (R3) applies. Therefore $f$ gives out charge at most $\f{1}{2}$ for each incident $6^+$-vertex, and at most $1.1$ for each incident $5$-vertex.

If $f$ is a $7^+$-face, then the initial charge $2d(f)-6$ is greater than $1.1d(f)$, so trivially $f$ must finish with non-negative charge. If $f$ is a $6$-face, then it has initial charge $6$, but it cannot have more than four incident $5$-vertices: otherwise, one $5$-vertex would be adjacent to two others, contradicting Claim \ref{5neighbour}. Note that this is still true even if the boundary of $f$ consists of two disjoint $3$-cycles.  Therefore, by (R1) and (R3), $f$ gives out charge at most $4.4 + 1 < 6$, so it finishes with positive charge.

If $f$ is a $5$-face, then it has initial charge $4$, and by Claim \ref{5neighbour} it has at most three incident $5$-vertices. If it has exactly three, then the two remaining vertices are not adjacent on $f$ and therefore (R3) does not apply; thus $f$ gives out charge at most $3.3$. If instead $f$ has at most two incident $5$-vertices, then it gives out charge at most $2.2 + \f{3}{2} < 4$, so in either case $f$ finishes with positive charge.

If $f$ is a $4$-face, then it begins with charge $2$ and has at most two incident $5$-vertices by Claim \ref{5neighbour}. First suppose $f$ has exactly two $5$-vertices. If they are adjacent, then (R2) implies that they each receive charge $\f{3}{4}$. The remaining two vertices of $f$ are $6^+$-vertices, so the greatest additional charge that $f$ can give out is $\f{1}{2}$, by (R3). Thus $f$ gives out charge at most $2$, as required. If instead the two $5$-vertices are not adjacent, then they each receive charge $1$ by (R2), but (R3) does not apply and so $f$ again gives out total charge $2$.

Now consider the case where $f$ is a $4$-face with exactly one incident $5$-vertex. This vertex receives charge $1$, and (R3) applies to at most two edges of $f$, so it gives out total charge at most $2$. Finally, if every vertex of $f$ is a $6^+$-vertex, then it gives out charge at most $\f{1}{2}$ for each vertex, and so we are done.
\end{proof}

Our graph $G$ begins with total charge at most $0$ and the rules preserve charge, so if $G$ has minimum degree $5$, Claim \ref{chargeswork} implies that some $5$-vertex must finish with charge at most $0$. If instead $\delta(G) = 6$, then as noted earlier, $G$ is a $6$-regular triangulation, and thus every vertex and face start and end with charge $0$. The discharging method therefore does not help when $\delta(G) = 6$, and we will have to treat this case separately.

For the remainder of this section, we will restrict ourselves to the case where $\delta(G) = 5$. Let $v$ be a $5$-vertex that has charge at most $0$ after the discharging process: in other words, $v$ receives total charge at most $1$ during discharging (recall that $v$ cannot give out charge). Let the neighbours of $v$ be $v_1$, $v_2$, $v_3$, $v_4$ and $v_5$ in anticlockwise order.

\begin{claim}
The five faces around $v$ are all different.
\end{claim}

\begin{proof}
Suppose that two of the faces around $v$ are the same face $f$. This face clearly must be a $4^+$-face. By the discharging rules, $f$ sends charge at least $\f{3}{4}$ to $v$ for each appearance of $v$ on the boundary of $f$, so $f$ sends total charge at least $\f{3}{2}$ to $v$, a contradiction.
\end{proof}

We will make use the following lemma to eliminate the cases in the remainder of the proof.

\begin{lemma}\label{4vlemma}
Suppose that the edges $v_1v_2$, $v_2v_3$ and $v_3v_4$ are all present in $G$, and that $v_2$ and $v_3$ are $6$-vertices with a common neighbour $x \neq v$. Let $c$ be a nice colouring of $G \setminus \{v\}$ in which $c(v_1)$, $c(v_2)$, $c(v_3)$ and $c(v_4)$ are all distinct. Then $v_2$ and $v_3$ do not forbid two distinct colours at $v$ by oddness unless at least one is already forbidden by properness.
\end{lemma}

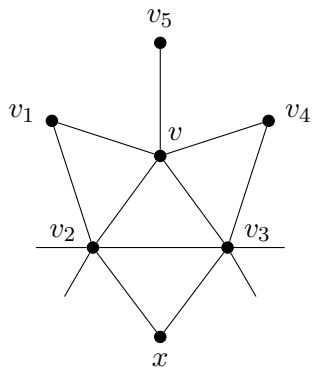
\begin{figure}[ht]
    \centering
\begin{tikzpicture}[scale=1.5]
\tikzset{enclosed/.style={draw, circle, inner sep=0pt, minimum size=.15cm, fill=black}}
\node[enclosed, label={above,xshift=2mm: $v$}] (v) at (0,0) {};
\node[enclosed, label={above: $v_5$}] (v5) at (0,1) {};
\node[enclosed, label={right,yshift=1mm: $v_4$}] (v4) at (0.95,0.31) {};
\node[enclosed, label={right,yshift=2mm: $v_3$}] (v3) at (0.59,-0.81) {};
\node[enclosed, label={left,yshift=2mm: $v_2$}] (v2) at (-0.59,-0.81) {};
\node[enclosed, label={left,yshift=1mm: $v_1$}] (v1) at (-0.95,0.31) {};
\node[enclosed, label={below: $x$}] (x) at (0,-1.6) {};

\draw (v) -- (v1) {};
\draw (v) -- (v2) {};
\draw (v) -- (v3) {};
\draw (v) -- (v4) {};
\draw (v) -- (v5) {};
\draw (v1) -- (v2) -- (v3) -- (v4) {};
\draw (v2) -- (x) -- (v3) {};
\draw (v2) -- ++(180:5mm) {};
\draw (v2) -- ++(240:5mm) {};
\draw (v3) -- ++(0:5mm) {};
\draw (v3) -- ++(300:5mm) {};
\end{tikzpicture}
    \caption{The setting of Lemma \ref{4vlemma}. All edges from $v$, $v_2$ and $v_3$ are shown.}
    \label{4vlemmafig}
\end{figure}

\begin{proof}
Suppose that a colouring $c$ exists as above, and let $c(v_i) = i$ for $i = 1,2,3,4$. Suppose further that $v_2$ and $v_3$ forbid colours $2'$ and $3'$ respectively at $v$ by oddness, where $2' \neq 3'$ and $2',3' \notin \{1,2,3,4\}$. Now the colours of the neighbours of $v_2$ in $G \setminus \{v\}$ must be $1,1,3,3,2'$ in some order. In particular, $c(x) \in \{1,3,2'\}$. Similarly, consideration of the colours around $v_3$ shows that $c(x) \in \{2,4,3'\}$. But now there is no possible colour for $x$ and we have a contradiction.
\end{proof}

We will need one more structural lemma.

\begin{lemma}\label{adjacent5}
$G$ does not contain two adjacent $5$-vertices which have two common neighbours.
\end{lemma}

\begin{figure}[ht]
    \centering
\begin{tikzpicture}[scale=1.5]
\tikzset{enclosed/.style={draw, circle, inner sep=0pt, minimum size=.15cm, fill=black}}
\node[enclosed, label={left: $u$}] (u) at (-0.5,0) {};
\node[enclosed, label={right: $v$}] (v) at (0.5,0) {};
\node[enclosed, label={above: $x$}] (x) at (0,0.86) {};
\node[enclosed, label={below: $y$}] (y) at (0,-0.86) {};
\node[enclosed, label={above left,xshift=0.5mm,yshift=-0.5mm: $w_1$}] (w1) at (-1.21,0.71) {};
\node[enclosed, label={below left,xshift=0.5mm,yshift=0.5mm: $w_2$}] (w2) at (-1.21,-0.71) {};
\node[enclosed, label={above right,xshift=-0.5mm,yshift=-0.5mm: $z_1$}] (z1) at (1.21,0.71) {};
\node[enclosed, label={below right,xshift=-0.5mm,yshift=0.5mm: $z_2$}] (z2) at (1.21,-0.71) {};

\draw (u) -- (v) {};
\draw (w1) -- (u) -- (x) -- (v) -- (z1) {};
\draw (w2) -- (u) -- (y) -- (v) -- (z2) {};
\end{tikzpicture}
    \caption{The setting of Lemma \ref{adjacent5}. All edges from $u$ and $v$ are shown.}
    \label{adjacent5fig}
\end{figure}

\begin{proof}
Suppose to the contrary that $u$ and $v$ are adjacent $5$-vertices with common neighbours $x$ and $y$. Let the two remaining neighbours of $u$ be $w_1$ and $w_2$, and let the two remaining neighbours of $v$ be $z_1$ and $z_2$. Note that we could have some $w_i = z_j$.

Consider the graph $G' = G/\{u, v\}$ formed by contracting the edge $uv$ to form a single vertex $\{u, v\}$. By the minimality of $G$, this graph has a nice colouring, which we will call $c$. We will use $c$ to colour the vertices of $G \setminus \{u,v\}$, and show that we can always choose colours at $u$ and $v$ to produce an odd colouring of $G$. Note that any colouring of $G$ will be odd at $u$ and $v$, since they both have odd degree.

Suppose that we colour $u$ with $c(\{u,v\})$. The resulting partial colouring is both proper and odd at each $w_i$ that is not the same as some $z_j$, since $c$ is an odd colouring of $G'$. Note that $u$ has degree $5$ so cannot forbid a colour at $v$ by oddness. For every remaining neighbour $t$ of $v$, let $b(t)$ be the colour that $t$ forbids at $v$ by oddness, if it exists. There are at most $5$ colours forbidden at $v$ by properness, and at most $4$ forbidden by oddness. Therefore there is always a colour available at $v$ to produce an odd colouring of $G$ unless the $9$ colours $c(\{u,v\}), c(x), c(y), c(z_1), c(z_2), b(x), b(y), b(z_1), b(z_2)$ are all distinct. Similarly, we could colour $v$ with $c(\{u,v\})$ instead of $u$, and define $b$ as above; there is no ambiguity in the definitions of $b(x)$ and $b(y)$ since the neighbourhoods of $x$ and $y$ have the same multisets of colours in each case. This partial colouring extends to an odd colouring of $G$ unless $c(\{u,v\}), c(x), c(y), c(w_1), c(w_2), b(x), b(y), b(w_1), b(w_2)$ are all distinct.

Suppose that the two sets of $9$ colours above are indeed distinct. Now we assign the colour $b(x)$ to $u$ and $b(y)$ to $v$. This results in a proper colouring of $G$ by distinctness. This colouring is odd at $w_i$ and $z_j$, also by distinctness. Finally, $c(\{u,v\})$ and $b(y)$ each appear an odd number of times in $N(x)$, since they are distinct from each other and from $b(x)$; similarly $c(\{u,v\})$ and $b(x)$ appear an odd number of times in $N(y)$. Thus $G$ has a nice colouring, which is a contradiction.
\end{proof}

\begin{prop}\label{3faces}
The five faces around $v$ are all $3$-faces.
\end{prop}

\begin{proof}
First note that, by (R1), a $5$-face would send charge greater than $1$ to $v$, which is impossible. As noted previously, a $4$-face sends charge at least $\f{3}{4}$, so $v$ can be incident to at most one $4$-face. Let the $4$-face be $v_5vv_1w$. Now Lemma \ref{adjacent5} implies that $v_2$, $v_3$ and $v_4$ are all $6^+$-vertices. 

\begin{figure}[ht]
    \centering
\begin{tikzpicture}[scale=1.5]
\tikzset{enclosed/.style={draw, circle, inner sep=0pt, minimum size=.15cm, fill=black}}
\node[enclosed, label={above,xshift=2mm: $v$}] (v) at (0,0) {};
\node[enclosed, label={above: $v_5$}] (v5) at (0,1) {};
\node[enclosed, label={right: $v_4$}] (v4) at (0.95,0.31) {};
\node[enclosed, label={below right,xshift=-0.5mm,yshift=0.3mm: $v_3$}] (v3) at (0.59,-0.81) {};
\node[enclosed, label={below left,xshift=0.5mm,yshift=0.3mm: $v_2$}] (v2) at (-0.59,-0.81) {};
\node[enclosed, label={left: $v_1$}] (v1) at (-0.95,0.31) {};
\node[enclosed, label={above left,xshift=0.5mm,yshift=-0.5mm: $w$}] (w) at (-0.75,1) {};

\draw (v) -- (v1) {};
\draw (v) -- (v2) {};
\draw (v) -- (v3) {};
\draw (v) -- (v4) {};
\draw (v) -- (v5) {};
\draw (v1) -- (v2) -- (v3) -- (v4) -- (v5) -- (w) -- (v1) {};
\end{tikzpicture}
    \label{one4facefig}
    \caption{$v$ and its surrounding faces}
\end{figure}
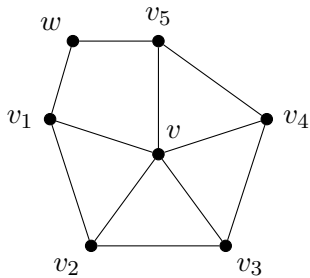

Let $c$ be a nice colouring of $G\setminus \{v\}$, which must exist by the minimality of $G$. We now divide into $3$ cases.
\begin{enumerate}
    \item $v$ has a neighbouring $5$-vertex.
    \item All the neighbours of $v$ are $6^+$-vertices, and $|c(N(v))| = 4$.
    \item All the neighbours of $v$ are $6^+$-vertices, and $|c(N(v))| = 5$.
\end{enumerate}
Note that we must have $|c(N(v))| \geq 4$, otherwise there would be at most $8$ colours forbidden at $v$ and we could extend $c$ to a nice colouring of $G$. Henceforth, if the edge $v_iv_{i+1}$ is present, we will refer to the face adjacent to $vv_iv_{i+1}$ along $v_iv_{i+1}$ as the \emph{external face along $v_iv_{i+1}$}.

\textbf{Case 1:} $v$ has a neighbouring $5$-vertex.\\
Without loss of generality, the $5$-vertex is $v_1$, and so $v_5$ is a $6^+$-vertex. Now $v_1$ does not forbid a colour at $v$ by oddness, so we must have $|c(N(v))| = 5$, otherwise there would be at most $8$ colours forbidden at $v$.

Note that a vertex of degree $d \geq 7$ can have at most $\left\lfloor\f{d}{2}\right\rfloor$ neighbouring blocks in the terminology of (R4). Since the total charge given out is $d-6$, this implies that a $7$-vertex gives out charge at least $\f{1}{3}$ to every block, and an $8^+$-vertex gives out charge at least $\f{1}{2}$.

If either $v_3$ or $v_4$ is a $7^+$-vertex, then by (R4) it sends charge at least $\f{1}{3}$ to $v$, since $v$ is a singleton block. This is a contradiction because $v$ already receives charge at least $\f{3}{4}$ from the $4$-face. Hence $v_3$ and $v_4$ are both $6$-vertices.

By (R3), the external face along $v_3v_4$ cannot be a $4^+$-face, otherwise it would send charge $\f{1}{2}$ to $v$. Hence it is a $3$-face, and so $v_3$ and $v_4$ have a common neighbour that is not $v$. Now we can apply Lemma \ref{4vlemma} to $v_2$, $v_3$, $v_4$, $v_5$, which implies that $v_3$ and $v_4$ cannot forbid distinct colours at $v$ by oddness that are not already in $c(N(v))$. Hence at most $8$ colours are forbidden at $v$ and we can extend $c$ to $v$. This completes Case 1.

\textbf{Case 2:} All the neighbours of $v$ are $6^+$-vertices, and $|c(N(v))| = 4$.\\
The external faces along $v_1v_2$, $v_2v_3$, $v_3v_4$ and $v_4v_5$ must all be $3$-faces, otherwise they would give charge $\f{1}{2}$ to $v$ by (R3). Denote the common neighbour of $v_2$ and $v_3$ on this external $3$-face by $x$, and the corresponding common neighbour of $v_3$ and $v_4$ by $y$. 

As in Case 1, if any of $v_2$, $v_3$, $v_4$ is a $7^+$-vertex then it sends charge at least $\f{1}{3}$ to $v$ by (R4). Hence $v_2$, $v_3$, $v_4$ are all $6$-vertices. Lemma \ref{4vlemma} now applies to either $\{v_1, v_2, v_3, v_4\}$ or $\{v_2, v_3, v_4, v_5\}$, giving a contradiction, unless $c(v_2) = c(v_4)$. Let $c(v_i) = i$ for $i = 1,2,3,5$.

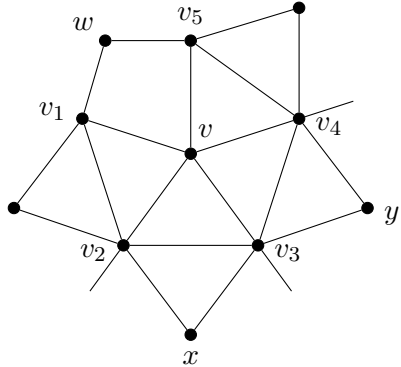
\begin{figure}[ht]
    \centering
\begin{tikzpicture}[scale=1.5]
\tikzset{enclosed/.style={draw, circle, inner sep=0pt, minimum size=.15cm, fill=black}}
\node[enclosed, label={above,xshift=2mm: $v$}] (v) at (0,0) {};
\node[enclosed, label={above: $v_5$}] (v5) at (0,1) {};
\node[enclosed, label={right,yshift=-1mm: $v_4$}] (v4) at (0.95,0.31) {};
\node[enclosed, label={right,yshift=-1mm: $v_3$}] (v3) at (0.59,-0.81) {};
\node[enclosed, label={left,yshift=-1mm: $v_2$}] (v2) at (-0.59,-0.81) {};
\node[enclosed, label={left,yshift=1mm: $v_1$}] (v1) at (-0.95,0.31) {};
\node[enclosed, label={above left,xshift=0.5mm,yshift=-0.5mm: $w$}] (w) at (-0.75,1) {};
\node[enclosed] (x1) at (-1.55,-0.48) {};
\node[enclosed, label={below: $x$}] (x) at (0,-1.6) {};
\node[enclosed, label={right,yshift=-1mm: $y$}] (y) at (1.55,-0.48) {};
\node[enclosed] (x4) at (0.95,1.29) {};

\draw (v) -- (v1) {};
\draw (v) -- (v2) {};
\draw (v) -- (v3) {};
\draw (v) -- (v4) {};
\draw (v) -- (v5) {};
\draw (v1) -- (v2) -- (v3) -- (v4) -- (v5) -- (w) -- (v1) {};
\draw (v1) -- (x1) -- (v2) -- (x) -- (v3) -- (y) -- (v4) -- (x4) -- (v5) {};
\draw (v2) -- ++(234:5mm) {};
\draw (v3) -- ++(306:5mm) {};
\draw (v4) -- ++(18:5mm) {};
\end{tikzpicture}
    \label{4facecase2fig}
    \caption{The setting of Case 2. All edges from $v$, $v_2$, $v_3$ and $v_4$ are shown.}
\end{figure}

Consider the colours in $N(v_2)\setminus\{v\}$. $v_2$ must forbid a colour $2' \notin \{1,2,3,5\}$ at $v$ by oddness, so the vertices of $N(v_2)\setminus\{v\}$ must have colours $1,1,3,3,2'$ in some order. Hence $c(x) \in \{1,2'\}$. Similarly $c(y) \in \{5,4'\}$, where $4'$ is defined analogously to $2'$.

But now $N(v_3)\setminus\{v\}$ has two vertices of colour $3$ as well as one of colour $1$ or $2'$ and one of colour $5$ or $4'$. There is therefore no way for $v_3$ to forbid a new colour $3'$ at $v$ by oddness, and so there are at most $8$ colours forbidden at $v$. This completes Case 2.

\textbf{Case 3:} All the neighbours of $v$ are $6^+$-vertices, and $|c(N(v))| = 5$.\\
First, note that the external faces along $v_1v_2$, $v_2v_3$, $v_3v_4$ and $v_4v_5$ must all be $3$-faces, as in Case 2. For $i = 1,2,3,4$, let $x_i$ be the common neighbour of $v_i$ and $v_{i+1}$ on this external $3$-face, and recall that $w$ is the final vertex of the $4$-face containing $v_1$, $v$ and $v_5$. Let $c(v_i) = i$, and let the colour that $v_i$ forbids at $v$ by oddness be $i'$, if it exists.

We also have that $v_2$, $v_3$ and $v_4$ must all be $6$-vertices, since if any were $7^+$-vertices, they would give charge at least $\f{1}{3}$ to $v$ by (R4). In addition, if either $v_1$ or $v_5$ is a $7^+$-vertex, then it also gives charge at least $\f{1}{3}$ to $v$ except in the event that $\{v\}$ is not a block; this only happens if $w$ is a $5$-vertex. However, in this case the face $vv_1wv_5$ gives charge $1$ to $v$ by (R2), and the $7^+$-vertex still sends some positive charge to $v$, which is a contradiction. Hence all the neighbours of $v$ are $6$-vertices.

\begin{figure}[ht]
    \centering
\begin{tikzpicture}[scale=1.5]
\tikzset{enclosed/.style={draw, circle, inner sep=0pt, minimum size=.15cm, fill=black}}
\node[enclosed, label={above,xshift=2mm: $v$}] (v) at (0,0) {};
\node[enclosed, label={above left: $v_5$}] (v5) at (0,1) {};
\node[enclosed, label={right,yshift=-1mm: $v_4$}] (v4) at (0.95,0.31) {};
\node[enclosed, label={right,yshift=-1mm: $v_3$}] (v3) at (0.59,-0.81) {};
\node[enclosed, label={left,yshift=-1mm: $v_2$}] (v2) at (-0.59,-0.81) {};
\node[enclosed, label={left,yshift=-1mm: $v_1$}] (v1) at (-0.95,0.31) {};
\node[enclosed, label={above left,xshift=0.5mm,yshift=-0.5mm: $w$}] (w) at (-0.75,1) {};
\node[enclosed, label={left,yshift=-1mm: $x_1$}] (x1) at (-1.55,-0.48) {};
\node[enclosed, label={below: $x_2$}] (x2) at (0,-1.6) {};
\node[enclosed, label={right,yshift=-1mm: $x_3$}] (x3) at (1.55,-0.48) {};
\node[enclosed, label={right,yshift=2mm: $x_4$}] (x4) at (0.95,1.29) {};

\draw (v) -- (v1) {};
\draw (v) -- (v2) {};
\draw (v) -- (v3) {};
\draw (v) -- (v4) {};
\draw (v) -- (v5) {};
\draw (v1) -- (v2) -- (v3) -- (v4) -- (v5) -- (w) -- (v1) {};
\draw (v1) -- (x1) -- (v2) -- (x) -- (v3) -- (y) -- (v4) -- (x4) -- (v5) {};
\draw (v1) -- ++(162:5mm) {};
\draw (v2) -- ++(234:5mm) {};
\draw (v3) -- ++(306:5mm) {};
\draw (v4) -- ++(18:5mm) {};
\draw (v5) -- ++(90:5mm) {};
\end{tikzpicture}
    \label{4facecase3fig}
    \caption{The setting of Case 3. All edges from $v$ and each $v_i$ are shown.}
\end{figure}
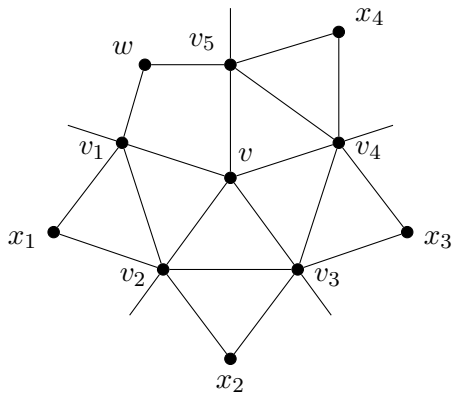

Lemma \ref{4vlemma} applies to both $\{v_1, v_2, v_3, v_4\}$ and $\{v_2, v_3, v_4, v_5\}$, and there must be at least $4$ colours forbidden by oddness at $v$. Thus the only possibility is that these are $1',2',4',5'$, while $v_3$ does not forbid a colour by oddness at $v$ that is not already forbidden in some other way.

The colours of the vertices of $N(v_2)\setminus\{v\}$ must be $1,1,3,3,2'$, and so $c(x_2) \in \{1,2'\}$. Similarly $c(x_3) \in \{5,4'\}$. Let $c'$ be the restriction of $c$ to $G\setminus\{v,v_2\}$. We will extend $c'$ to a nice colouring of $G$. 

First, assign colour $2$ to $v$. This produces a proper partial colouring that is odd at $v_4$ and $v_5$. The colouring is also odd at $v_2$, since $v_2$ has exactly one neighbour of colour $2$. In addition, the colours $2,4,c(x_2),c(x_3)$ are all distinct and all already used in $N(v_3)$, so $v_3$ cannot forbid a colour at $v_2$ by oddness. Therefore there are at most $8$ colours forbidden at $v_2$: the $4$ colours of its neighbours, and at most $4$ colours forbidden by oddness, since neither $v$ nor $v_3$ forbids a colour by oddness. Hence there is a colour left over that we can use to colour $v_2$. This finishes the proof of Proposition \ref{3faces}.
\end{proof}

We now know that all the faces around $v$ are $3$-faces. Lemma \ref{adjacent5} therefore implies that all the $v_i$ are $6^+$-vertices. As before, let $c$ be a nice colouring of $G\setminus\{v\}$. Once again, we must have $|c(N(v))| \geq 4$, so we first rule out the case $|c(N(v))| = 4$.

\begin{prop}\label{not4colours}
We have $|c(N(v))| = 5$.
\end{prop}

\begin{proof}
Suppose that $|c(N(v))| = 4$. Without loss of generality, we have $c(v_i) = i$ for $i = 1,2,3,5$ and $c(v_4) = 2$. Since all $9$ colours must be forbidden at $v$ by either properness or oddness, every $v_i$ forbids a colour by oddness, which we call $i'$ as before. Thus every $v_i$ has even degree.

By Lemma \ref{4vlemma}, if $v_1$ and $v_2$ are both $6$-vertices and the external face along $v_1v_2$ is a $3$-face, then one of $v_1$ and $v_2$ does not forbid a distinct colour by oddness, which is a contradiction. Hence either the external face along $v_1v_2$ is a $4^+$-face, giving charge $\f{1}{2}$ to $v$ by (R3), or one of $v_1$ and $v_2$ is an $8^+$-vertex, giving charge at least $\f{1}{2}$ to $v$ by (R4). Similarly, either the external face along $v_4v_5$ is a $4^+$-face or one of $v_4$ and $v_5$ is an $8^+$-vertex. Together these give a total charge of at least $1$ to $v$. This implies that $v$ cannot receive any further charge: in particular, the external faces along $v_2v_3$ and $v_3v_4$ are $3$-faces, and $v_3$ has degree $6$. Let $x$ be the common neighbour of $v_2$ and $v_3$ on this external face, and let $y$ be the corresponding common neighbour of $v_3$ and $v_4$.

\begin{figure}[ht]
    \centering
\begin{tikzpicture}[scale=1.5]
\tikzset{enclosed/.style={draw, circle, inner sep=0pt, minimum size=.15cm, fill=black}}
\node[enclosed, label={above,xshift=2mm: $v$}] (v) at (0,0) {};
\node[enclosed, label={above: $v_5$}] (v5) at (0,1) {};
\node[enclosed, label={right,yshift=1mm: $v_4$}] (v4) at (0.95,0.31) {};
\node[enclosed, label={right,yshift=-1mm: $v_3$}] (v3) at (0.59,-0.81) {};
\node[enclosed, label={left,yshift=-1mm: $v_2$}] (v2) at (-0.59,-0.81) {};
\node[enclosed, label={left,yshift=1mm: $v_1$}] (v1) at (-0.95,0.31) {};
\node[enclosed, label={below: $x$}] (x) at (0,-1.6) {};
\node[enclosed, label={right,yshift=-1mm: $y$}] (y) at (1.55,-0.48) {};

\draw (v) -- (v1) {};
\draw (v) -- (v2) {};
\draw (v) -- (v3) {};
\draw (v) -- (v4) {};
\draw (v) -- (v5) {};
\draw (v1) -- (v2) -- (v3) -- (v4) -- (v5) -- (v1) {};
\draw (v2) -- (x) -- (v3) -- (y) -- (v4) {};
\draw (v3) -- ++(306:5mm) {};
\end{tikzpicture}
    \label{not4coloursfig}
    \caption{The case $|c(N(v))| = 4$. All edges from $v$ and $v_3$ are shown.}
\end{figure}
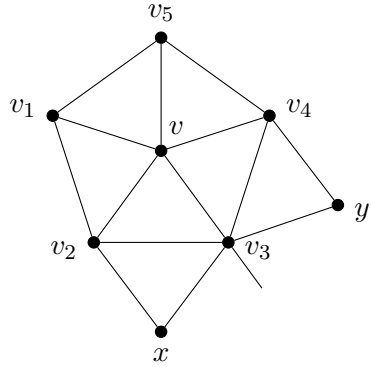

Suppose first that $v_2$ and $v_4$ are both $6$-vertices. Then the colours in $N(v_2)\setminus\{v\}$ are $1,1,3,3,2'$ in some order, so $c(x) \in \{1,2'\}$. Similarly $c(y) \in \{5,4'\}$. But now $v_3$ cannot forbid a new colour $3'$ by oddness. So at least one of $v_2$ and $v_4$ must be an $8^+$-vertex. Without loss of generality it is $v_2$.

Note that $v_2$ must send charge exactly $\f{1}{2}$ to $v$, since otherwise $v$ would receive total charge greater than $1$. The only way this can happen is if $v_2$ has degree exactly $8$ and the neighbours of $v_2$ are alternately $5$-vertices and $6^+$-vertices. This implies that $x$ is a $5$-vertex.

Now consider the restriction $c'$ of $c$ to $G\setminus\{v,v_3\}$. We will extend this to a nice colouring of $G$. First we assign the colour $3$ to $v$. This produces a proper partial colouring that is odd at $v_1$ and $v_5$; it is also odd at $v_3$, since $v_3$ has exactly one neighbour of colour $3$. Since $v_3$ forbids a colour by oddness at $v$ in the colouring $c$, we must have $|c(N(v_3)\setminus\{v\})| \leq 3$, and hence after colouring $v$ we have used at most $4$ colours in $N(v_3)$. In addition, $v$ and $x$ both have odd degree, so there at most $4$ colours forbidden by oddness at $v_3$. There is therefore a colour left over to use at $v_3$, and we are done.
\end{proof}

We are now ready to finish the proof for the case $\delta(G) = 5$. Let $c$ be a nice colouring of $G\setminus\{v\}$. Since we know $|c(N(v))| = 5$, the $v_i$ must together forbid the remaining $4$ colours by oddness. Therefore, without loss of generality, we have that for $i = 1,2,4,5$, $v_i$ forbids a colour $i' \notin \{1,2,3,4,5\}$, where all the $i'$ are distinct. This implies that all the $v_i$ except possibly $v_3$ have even degree.

If the external face along $v_5v_1$ is a $3$-face and $v_1$ and $v_5$ are both $6$-vertices, then Lemma \ref{4vlemma} applied to $\{v_4, v_5, v_1, v_2\}$ gives a contradiction, since $1'$ and $5'$ are distinct from each other and all the other colours. Therefore either the external face along $v_5v_1$ is a $4^+$-face, or one of $v_1$ and $v_5$ is an $8^+$-vertex. In either case, charge at least $\f{1}{2}$ is sent to $v$.

We can similarly apply Lemma \ref{4vlemma} to $\{v_3, v_4, v_5, v_1\}$ and $\{v_5, v_1, v_2, v_3\}$. This shows that either the external face along $v_4v_5$ is a $4^+$-face or one of $v_4$ and $v_5$ is an $8^+$-vertex, and that either the external face along $v_1v_2$ is a $4^+$-face or one of $v_1$ and $v_2$ is an $8^+$-vertex.

In particular, if both $v_1$ and $v_5$ are $6$-vertices, then this implies that each of the edges $v_1v_2$ and $v_4v_5$ must either have an external $4^+$-face or contain an $8^+$-vertex. Together, these send charge at least $1$ to $v$, implying that the external face along $v_1v_5$ cannot be a $4^+$-face. But this is a contradiction, since we have just shown that if $v_1$ and $v_5$ are both $6$-vertices then the external face along $v_5v_1$ is a $4^+$-face. Hence at least one of $v_1$ and $v_5$ is an $8^+$-vertex.

Suppose without loss of generality that $v_1$ is an $8^+$-vertex. We still require that either the external face along $v_4v_5$ is a $4^+$-face or one of $v_4$ and $v_5$ is an $8^+$-vertex. In either case, $v$ receives total charge at least $1$. As in the proof of Proposition \ref{not4colours}, in order for the total charge to be exactly $1$, we must have that $v_1$ has degree exactly $8$, $v_2$ is a $6$-vertex, and the external face along $v_1v_2$ is a $3$-face. Let $x$ be the common neighbour of $v_1$ and $v_2$ on this face. Again, as in the proof of Proposition \ref{not4colours}, the neighbours of $v_1$ must alternate between $5$-vertices and $6^+$-vertices, so $x$ has degree $5$.

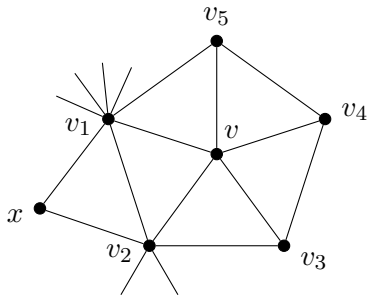
\begin{figure}[ht]
    \centering
\begin{tikzpicture}
[scale=1.5]
\tikzset{enclosed/.style={draw, circle, inner sep=0pt, minimum size=.15cm, fill=black}}
\node[enclosed, label={above,xshift=2mm: $v$}] (v) at (0,0) {};
\node[enclosed, label={above: $v_5$}] (v5) at (0,1) {};
\node[enclosed, label={right,yshift=1mm: $v_4$}] (v4) at (0.95,0.31) {};
\node[enclosed, label={right,yshift=-2mm: $v_3$}] (v3) at (0.59,-0.81) {};
\node[enclosed, label={left,yshift=-1mm: $v_2$}] (v2) at (-0.59,-0.81) {};
\node[enclosed, label={left,yshift=-1mm: $v_1$}] (v1) at (-0.95,0.31) {};
\node[enclosed, label={left,yshift=-1mm: $x$}] (x) at (-1.55,-0.48) {};

\draw (v) -- (v1) {};
\draw (v) -- (v2) {};
\draw (v) -- (v3) {};
\draw (v) -- (v4) {};
\draw (v) -- (v5) {};
\draw (v1) -- (v2) -- (v3) -- (v4) -- (v5) -- (v1) {};
\draw (v1) -- (x) -- (v2) {};
\draw (v2) -- ++(240:5mm) {};
\draw (v2) -- ++(300:5mm) {};
\draw (v1) -- ++(66:5mm) {};
\draw (v1) -- ++(96:5mm) {};
\draw (v1) -- ++(126:5mm) {};
\draw (v1) -- ++(156:5mm) {};
\end{tikzpicture}
    \label{finalcasefig}
    \caption{The case $|c(N(v))| = 5$, where $v_1$ is an $8$-vertex without loss of generality. All edges from $v$, $v_1$ and $v_2$ are shown.}
\end{figure}

Now we restrict $c$ to a colouring $c'$ of $G\setminus\{v,v_2\}$, and extend this to a nice colouring of $G$. First we assign colour $2$ to $v$. This produces a proper partial colouring which is odd at $v_4$ and $v_5$. The vertices of $N(v_2)\setminus\{v\}$ have colours $1,1,3,3,2'$, so after colouring $v$, we have that $v$ is the only neighbour of $v_2$ with colour $2$, and hence the colouring is odd at $v_2$. In addition, there are $4$ colours forbidden at $v_2$ by properness, and since $v$ and $x$ have odd degree, there are at most $4$ colours forbidden by oddness. Hence there is a colour left over, which we use to colour $v_2$. This completes the proof for the case $\delta(G) = 5$.

\section{The case $\delta(G) = 6$}\label{mindeg6}

We are now left with the case $\delta(G) = 6$, which, as discussed earlier, corresponds to the case where $G$ is a $6$-regular triangulation of the torus. Such triangulations were classified by Altshuler \cite{altshuler}. We will use the notation of Balachandran and Sankarnarayanan \cite{bs21}. For $m, n \geq 1$ and $0 \leq t < n$, we define the graph $H = T(m,n,t)$ on vertex set $V(H) = \{(i,j): 1 \leq i \leq m, 1 \leq j \leq n\}$ as follows:

\begin{itemize}
    \item $(i,j) \sim (i,j+1)$ for all $i,j$,
    \item $(i,j) \sim (i+1,j),(i+1,j-1)$ for $1 \leq i < m$ and all $j$,
    \item $(m,j) \sim (1,j-t),(1,j-t-1)$ for all $j$.
\end{itemize}
The addition in the second co-ordinate is modulo $n$ above and throughout this section.

In other words, we begin with a grid graph of dimensions $(m+1) \times (n+1)$ and triangulate it. We then identify the top and bottom rows, and we identify the leftmost and rightmost columns with a shift of $t$ vertices.

\begin{figure}[ht]
    \centering
\begin{tikzpicture}[scale=0.8]
\tikzset{enclosed/.style={draw, circle, inner sep=0pt, minimum size=.15cm, fill=black}}
\draw (1,1) grid (5,7);
\foreach \x in {1,...,5} {%
    \foreach \y in {1,...,7} {%
        \node[enclosed] () at (\x,\y) {};
        }%
    }%
\foreach \y in {1,...,6} {%
    \node[left] at (1,\y) {\footnotesize{(1,\y)}};
    }%
\foreach \x in {2,3,4} {%
    \node[below] at (\x,1) {\footnotesize{(\x,1)}};
    \node[above] at (\x,7) {\footnotesize{(\x,1)}};
    }%
\node[left] at (1,7) {\footnotesize{(1,1)}};
\node[right] at (5,1) {\footnotesize{(1,3)}};
\node[right] at (5,2) {\footnotesize{(1,4)}};
\node[right] at (5,3) {\footnotesize{(1,5)}};
\node[right] at (5,4) {\footnotesize{(1,6)}};
\node[right] at (5,5) {\footnotesize{(1,1)}};
\node[right] at (5,6) {\footnotesize{(1,2)}};
\node[right] at (5,7) {\footnotesize{(1,3)}};

\draw (2,1) -- (1,2);
\draw (3,1) -- (1,3);
\draw (4,1) -- (1,4);
\draw (5,1) -- (1,5);
\draw (5,2) -- (1,6);
\draw (5,3) -- (1,7);
\draw (5,4) -- (2,7);
\draw (5,5) -- (3,7);
\draw (5,6) -- (4,7);
\end{tikzpicture}
    \caption{A diagram of $T(4,6,4)$. The first row and column are shown twice.}
    \label{torusexample}
\end{figure}
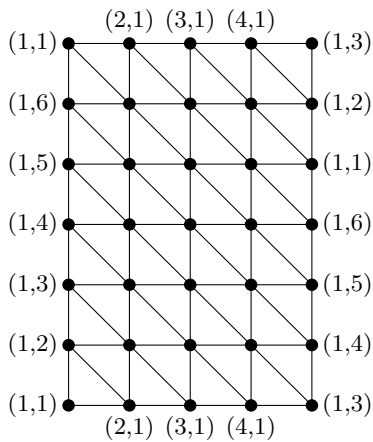

Note that two graphs $T(m,n,t)$ and $T(m',n',t')$ with different parameters can be isomorphic, and that this construction does not always produce a simple graph, although we are only concerned with the cases where it does.

\begin{thm}[Altshuler]\label{altshulerthm}
Every $6$-regular triangulation of the torus is isomorphic to $T(m,n,t)$ for some $m,n,t$.
\end{thm}

To finish the proof of Theorem \ref{mainthm}, it therefore suffices to prove that every $T(m,n,t)$ that is a simple graph admits an odd colouring with at most $9$ colours. We do not claim that this is optimal; indeed, we believe that with more care it should be possible to show that $7$ colours will always suffice. We will consider three cases: $m \geq 3$, $m=2$ and $m=1$. Note that we must have $n \geq 3$, or $T(m,n,t)$ would not be simple.

We will find it useful to partition the $9$ available colours into three classes $C_1 = \{1,2,3\}$, $C_2 = \{4,5,6\}$, $C_3 = \{7,8,9\}$. In diagrams, the colours of $C_1$ will be represented by three shades of red, with $1$ being darkest and $3$ lightest. Similarly, $C_2$ will be represented by three shades of blue, and $C_3$ by three shades of green.

\subsection*{$m \geq 3$}\hfill\\
We begin by using only one colour class in each column: in column $i$ we will use the class $C_r$ where $r \equiv i \pmod 3$, unless $m \equiv 1 \pmod 3$, in which case we will use $C_2$ in column $m$. This ensures that the same class of colours is not used in two neighbouring columns. We call a column \emph{bad} if its two neighbouring columns use the same class; otherwise it is \emph{good}. If $m \equiv 0 \pmod 3$ then there are no bad columns. If $m \equiv 1 \pmod 3$ then columns $1$ and $m-1$ are bad, and if $m \equiv 2 \pmod 3$ then columns $1$ and $m$ are bad.

Next, we apply the same construction within each column: we colour $(i,j)$ with the colour $s$ in the correct class that satisfies $s \equiv j \pmod 3$, unless $n \equiv 1 \pmod 3$, in which case we use the colour with $s \equiv 2 \pmod 3$ at $(i,n)$. We call the resulting colouring $c$; note that this initial colouring does not depend on the value of $t$. Again, we call a row \emph{bad} if its two neighbouring rows use the same set of colours; otherwise it is \emph{good}. The same classification of bad rows applies as for bad columns, according to the value of $n$ modulo $3$.

This construction clearly produces a proper colouring, since $n \geq 3$. In addition, for any vertex $(i,j)$ that is not in a bad column, the colours of $(i+1,j)$ and $(i+1,j-1)$ appear exactly once in its neighbourhood. If $(i,j)$ is in a bad column but not in a bad row, then the colours of $(i,j+1)$ and $(i,j-1)$ appear exactly once in its neighbourhood. Therefore the only vertices at which the colouring $c$ may not be odd are those that are both in a bad column and a bad row. These only exist if neither $m$ nor $n$ is $0 \pmod 3$, and in that case there are exactly $4$ of them. We call these \emph{bad vertices}.

First consider the case where $m$ and $n$ are both $1 \pmod 3$. Columns $1$ and $m-1$ and rows $1$ and $n-1$ are bad. Now, for each bad vertex $v = (i,j)$, its neighbour $w = (i+1,j-1)$ is in a good column. Therefore the two neighbours of $w$ in column $i+2$ use colours from the one class that is not used at $v$ or $w$. There is thus one colour in this class that does not appear in $c(N(w))$, and we now recolour $w$ with that colour. We do this for all four bad vertices, creating a new colouring $c'$. Let $S$ be the set consisting of the four recoloured vertices. By construction, $c'$ is a proper colouring, since none of the four vertices of $S$ are adjacent.

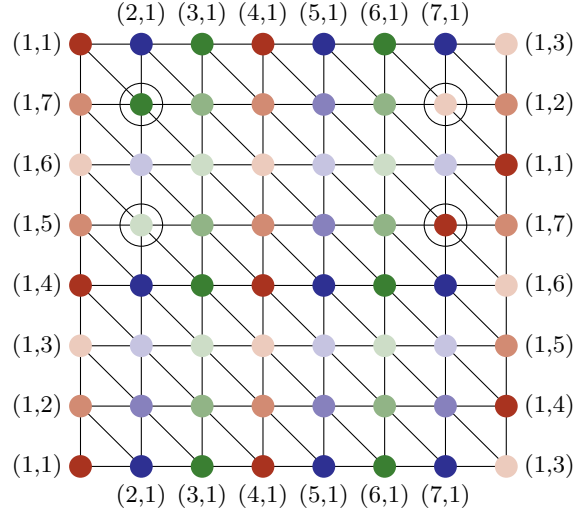
\begin{figure}[ht]
    \centering
\begin{tikzpicture}[scale=0.8]
\tikzset{col1/.style={fill, circle, inner sep=0pt, minimum size=3mm, color=Mahogany!100}}
\tikzset{col2/.style={fill, circle, inner sep=0pt, minimum size=3mm, color=Mahogany!50}}
\tikzset{col3/.style={fill, circle, inner sep=0pt, minimum size=3mm, color=Mahogany!20}}
\tikzset{col4/.style={fill, circle, inner sep=0pt, minimum size=3mm, color=Blue!100}}
\tikzset{col5/.style={fill, circle, inner sep=0pt, minimum size=3mm, color=Blue!50}}
\tikzset{col6/.style={fill, circle, inner sep=0pt, minimum size=3mm, color=Blue!20}}
\tikzset{col7/.style={fill, circle, inner sep=0pt, minimum size=3mm, color=OliveGreen!100}}
\tikzset{col8/.style={fill, circle, inner sep=0pt, minimum size=3mm, color=OliveGreen!50}}
\tikzset{col9/.style={fill, circle, inner sep=0pt, minimum size=3mm, color=OliveGreen!20}}

\foreach \y in {1,...,7} {%
    \node[left,xshift=-1mm] at (1,\y) {\footnotesize{(1,\y)}};
    }%
\foreach \y in {3,...,7} {%
    \node[right,xshift=1mm] at (8,\y-2) {\footnotesize{(1,\y)}};
    }%
\foreach \y in {1,2,3} {%
    \node[right,xshift=1mm] at (8,\y+5) {\footnotesize{(1,\y)}};
    }%
\node[left,xshift=-1mm] at (1,8) {\footnotesize{(1,1)}};
\foreach \x in {2,...,7} {%
    \node[below,yshift=-1mm] at (\x,1) {\footnotesize{(\x,1)}};
    \node[above,yshift=1mm] at (\x,8) {\footnotesize{(\x,1)}};
    }%

\draw (1,1) grid (8,8);
\draw (2,1) -- (1,2);
\draw (3,1) -- (1,3);
\draw (4,1) -- (1,4);
\draw (5,1) -- (1,5);
\draw (6,1) -- (1,6);
\draw (7,1) -- (1,7);
\draw (8,1) -- (1,8);
\draw (8,2) -- (2,8);
\draw (8,3) -- (3,8);
\draw (8,4) -- (4,8);
\draw (8,5) -- (5,8);
\draw (8,6) -- (6,8);
\draw (8,7) -- (7,8);

\draw (2,5) circle (0.35);
\draw (2,7) circle (0.35);
\draw (7,5) circle (0.35);
\draw (7,7) circle (0.35);

\node[col1] () at (1,1) {};
\node[col1] () at (1,4) {};
\node[col1] () at (1,8) {};
\node[col2] () at (1,2) {};
\node[col2] () at (1,5) {};
\node[col2] () at (1,7) {};
\node[col3] () at (1,3) {};
\node[col3] () at (1,6) {};
\node[col4] () at (2,1) {};
\node[col4] () at (2,4) {};
\node[col4] () at (2,8) {};
\node[col5] () at (2,2) {};
\node[col9] () at (2,5) {};
\node[col7] () at (2,7) {};
\node[col6] () at (2,3) {};
\node[col6] () at (2,6) {};
\node[col7] () at (3,1) {};
\node[col7] () at (3,4) {};
\node[col7] () at (3,8) {};
\node[col8] () at (3,2) {};
\node[col8] () at (3,5) {};
\node[col8] () at (3,7) {};
\node[col9] () at (3,3) {};
\node[col9] () at (3,6) {};
\node[col1] () at (4,1) {};
\node[col1] () at (4,4) {};
\node[col1] () at (4,8) {};
\node[col2] () at (4,2) {};
\node[col2] () at (4,5) {};
\node[col2] () at (4,7) {};
\node[col3] () at (4,3) {};
\node[col3] () at (4,6) {};
\node[col4] () at (5,1) {};
\node[col4] () at (5,4) {};
\node[col4] () at (5,8) {};
\node[col5] () at (5,2) {};
\node[col5] () at (5,5) {};
\node[col5] () at (5,7) {};
\node[col6] () at (5,3) {};
\node[col6] () at (5,6) {};
\node[col7] () at (6,1) {};
\node[col7] () at (6,4) {};
\node[col7] () at (6,8) {};
\node[col8] () at (6,2) {};
\node[col8] () at (6,5) {};
\node[col8] () at (6,7) {};
\node[col9] () at (6,3) {};
\node[col9] () at (6,6) {};
\node[col4] () at (7,1) {};
\node[col4] () at (7,4) {};
\node[col4] () at (7,8) {};
\node[col5] () at (7,2) {};
\node[col1] () at (7,5) {};
\node[col3] () at (7,7) {};
\node[col6] () at (7,3) {};
\node[col6] () at (7,6) {};
\node[col1] () at (8,6) {};
\node[col1] () at (8,2) {};
\node[col2] () at (8,7) {};
\node[col2] () at (8,3) {};
\node[col2] () at (8,5) {};
\node[col3] () at (8,1) {};
\node[col3] () at (8,4) {};
\node[col3] () at (8,8) {};
\end{tikzpicture}
    \caption{The case $m \equiv 1 \pmod 3$, $n \equiv 1 \pmod 3$, illustrated by $T(7,7,5)$. The vertices of $S$ are circled.}
    \label{m1n1}
\end{figure}

We now need to check that $c'$ is an odd colouring. First, observe that $c$ is odd at any vertex that has no neighbours in $S$, and therefore $c'$ is also odd at all such vertices. Note that in $c$, every vertex has an even number of neighbours in each colour class $C_i$. Therefore any vertex $v$ that is adjacent to exactly one vertex of $S$ has an odd number of neighbours in some colour class, and therefore $c'$ is odd at $v$.

There is no vertex adjacent to more than two vertices in $S$, so we are left to consider vertices with exactly two neighbours in $S$. There are two types of such vertices. First, there are vertices $v$ in bad columns with one neighbour in $S$ in each adjacent column. But then $v$ has exactly three neighbours in its own colour class in $c'$, so $c'$ is odd at $v$. Secondly, there are vertices $v = (i,j)$ in good columns where both $(i,j+1)$ and $(i,j-1)$ are in $S$. But in this case column $i-1$ is bad, and the two neighbours of $v$ in this column have distinct colours, each of which only appears once in $N(v)$. Thus $c'$ is an odd colouring.

Next we have the case $m \equiv 1 \pmod 3$, $n \equiv 2 \pmod 3$. The bad rows are rows $1$ and $n$, so $u = (2,n)$ is adjacent to the two bad vertices in column $1$ and $w = (m,n)$ is adjacent to the two bad vertices in column $m-1$. We recolour each of these vertices similarly to the previous case: we assign colour $9$ to $u$ since $9 \notin c(N(u))$, and recolour $w$ with whichever colour in $C_1$ does not appear in $c(N(w))$. This produces a proper colouring $c'$. The proof that $c'$ is an odd colouring proceeds just as in the $n \equiv 1 \pmod 3$ case, though it is simpler since only a vertex in a bad column can be adjacent to both $u$ and $v$.

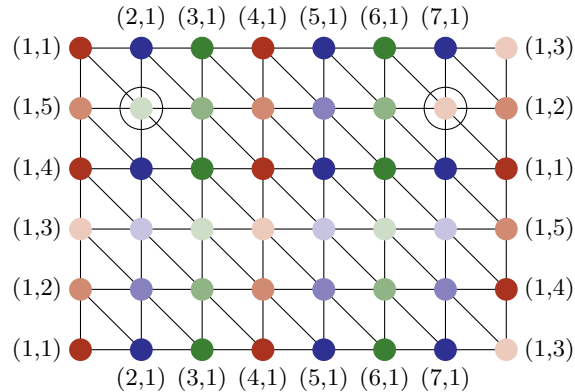
\begin{figure}[ht]
    \centering
\begin{tikzpicture}[scale=0.8]
\tikzset{col1/.style={fill, circle, inner sep=0pt, minimum size=3mm, color=Mahogany!100}}
\tikzset{col2/.style={fill, circle, inner sep=0pt, minimum size=3mm, color=Mahogany!50}}
\tikzset{col3/.style={fill, circle, inner sep=0pt, minimum size=3mm, color=Mahogany!20}}
\tikzset{col4/.style={fill, circle, inner sep=0pt, minimum size=3mm, color=Blue!100}}
\tikzset{col5/.style={fill, circle, inner sep=0pt, minimum size=3mm, color=Blue!50}}
\tikzset{col6/.style={fill, circle, inner sep=0pt, minimum size=3mm, color=Blue!20}}
\tikzset{col7/.style={fill, circle, inner sep=0pt, minimum size=3mm, color=OliveGreen!100}}
\tikzset{col8/.style={fill, circle, inner sep=0pt, minimum size=3mm, color=OliveGreen!50}}
\tikzset{col9/.style={fill, circle, inner sep=0pt, minimum size=3mm, color=OliveGreen!20}}

\foreach \y in {1,...,5} {%
    \node[left,xshift=-1mm] at (1,\y) {\footnotesize{(1,\y)}};
    }%
\foreach \y in {3,...,5} {%
    \node[right,xshift=1mm] at (8,\y-2) {\footnotesize{(1,\y)}};
    }%
\foreach \y in {1,2,3} {%
    \node[right,xshift=1mm] at (8,\y+3) {\footnotesize{(1,\y)}};
    }%
\node[left,xshift=-1mm] at (1,6) {\footnotesize{(1,1)}};
\foreach \x in {2,...,7} {%
    \node[below,yshift=-1mm] at (\x,1) {\footnotesize{(\x,1)}};
    \node[above,yshift=1mm] at (\x,6) {\footnotesize{(\x,1)}};
    }%

\draw (1,1) grid (8,6);
\draw (2,1) -- (1,2);
\draw (3,1) -- (1,3);
\draw (4,1) -- (1,4);
\draw (5,1) -- (1,5);
\draw (6,1) -- (1,6);
\draw (7,1) -- (2,6);
\draw (8,1) -- (3,6);
\draw (8,2) -- (4,6);
\draw (8,3) -- (5,6);
\draw (8,4) -- (6,6);
\draw (8,5) -- (7,6);

\draw (2,5) circle (0.35);
\draw (7,5) circle (0.35);

\node[col1] at (1,1) {};
\node[col1] at (1,4) {};
\node[col1] at (1,6) {};
\node[col2] at (1,2) {};
\node[col2] at (1,5) {};
\node[col3] at (1,3) {};
\node[col4] at (2,1) {};
\node[col4] at (2,4) {};
\node[col4] at (2,6) {};
\node[col5] at (2,2) {};
\node[col9] at (2,5) {};
\node[col6] at (2,3) {};
\node[col7] at (3,1) {};
\node[col7] at (3,4) {};
\node[col7] at (3,6) {};
\node[col8] at (3,2) {};
\node[col8] at (3,5) {};
\node[col9] at (3,3) {};
\node[col1] at (4,1) {};
\node[col1] at (4,4) {};
\node[col1] at (4,6) {};
\node[col2] at (4,2) {};
\node[col2] at (4,5) {};
\node[col3] at (4,3) {};
\node[col4] at (5,1) {};
\node[col4] at (5,4) {};
\node[col4] at (5,6) {};
\node[col5] at (5,2) {};
\node[col5] at (5,5) {};
\node[col6] at (5,3) {};
\node[col7] at (6,1) {};
\node[col7] at (6,4) {};
\node[col7] at (6,6) {};
\node[col8] at (6,2) {};
\node[col8] at (6,5) {};
\node[col9] at (6,3) {};
\node[col4] at (7,1) {};
\node[col4] at (7,4) {};
\node[col4] at (7,6) {};
\node[col5] at (7,2) {};
\node[col3] at (7,5) {};
\node[col6] at (7,3) {};
\node[col1] at (8,4) {};
\node[col1] at (8,2) {};
\node[col2] at (8,5) {};
\node[col2] at (8,3) {};
\node[col3] at (8,1) {};
\node[col3] at (8,6) {};
\end{tikzpicture}
    \caption{The case $m \equiv 1 \pmod 3$, $n \equiv 2 \pmod 3$, illustrated by $T(7,5,3)$. The two recoloured vertices $u$ and $w$ are circled.}
    \label{m1n2}
\end{figure}

We now move on to $m \equiv 2 \pmod 3$, with bad columns $1$ and $m$. Suppose that $n \equiv 1 \pmod 3$, so that rows $1$ and $n-1$ are bad. We recolour vertices as follows: $c'((2,n)) = c'((m-1,2)) = 7$, $c'((2,n-2)) = c'((m-1,n)) = 9$. Elsewhere $c' = c$. Let the set of recoloured vertices be $S$. This creates a proper colouring, and the proof that $c'$ is odd runs the same as in the previous cases, but with one additional detail: if $m = 5$ then it is possible for a vertex $v$ in column $3$ to have two neighbours in $S$, one in column $2$ and one in column $4$. However, in this case, $v$ has exactly one neighbour in colour class $C_1$ and one in class $C_2$, so $c'$ is odd at $v$.

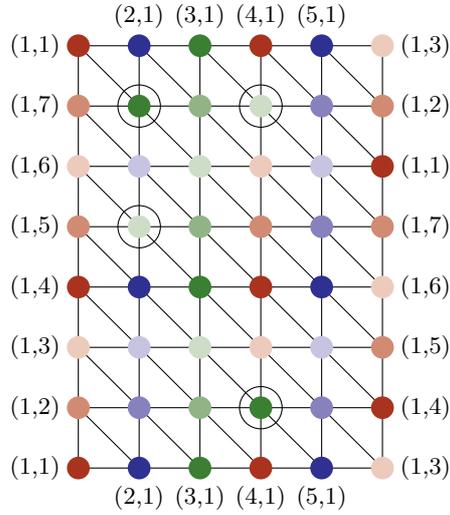
\begin{figure}[ht]
    \centering
\begin{tikzpicture}[scale=0.8]
\tikzset{col1/.style={fill, circle, inner sep=0pt, minimum size=3mm, color=Mahogany!100}}
\tikzset{col2/.style={fill, circle, inner sep=0pt, minimum size=3mm, color=Mahogany!50}}
\tikzset{col3/.style={fill, circle, inner sep=0pt, minimum size=3mm, color=Mahogany!20}}
\tikzset{col4/.style={fill, circle, inner sep=0pt, minimum size=3mm, color=Blue!100}}
\tikzset{col5/.style={fill, circle, inner sep=0pt, minimum size=3mm, color=Blue!50}}
\tikzset{col6/.style={fill, circle, inner sep=0pt, minimum size=3mm, color=Blue!20}}
\tikzset{col7/.style={fill, circle, inner sep=0pt, minimum size=3mm, color=OliveGreen!100}}
\tikzset{col8/.style={fill, circle, inner sep=0pt, minimum size=3mm, color=OliveGreen!50}}
\tikzset{col9/.style={fill, circle, inner sep=0pt, minimum size=3mm, color=OliveGreen!20}}

\foreach \y in {1,...,7} {%
    \node[left,xshift=-1mm] at (1,\y) {\footnotesize{(1,\y)}};
    }%
\foreach \y in {3,...,7} {%
    \node[right,xshift=1mm] at (6,\y-2) {\footnotesize{(1,\y)}};
    }%
\foreach \y in {1,2,3} {%
    \node[right,xshift=1mm] at (6,\y+5) {\footnotesize{(1,\y)}};
    }%
\node[left,xshift=-1mm] at (1,8) {\footnotesize{(1,1)}};
\foreach \x in {2,...,5} {%
    \node[below,yshift=-1mm] at (\x,1) {\footnotesize{(\x,1)}};
    \node[above,yshift=1mm] at (\x,8) {\footnotesize{(\x,1)}};
    }%

\draw (1,1) grid (6,8);
\draw (2,1) -- (1,2);
\draw (3,1) -- (1,3);
\draw (4,1) -- (1,4);
\draw (5,1) -- (1,5);
\draw (6,1) -- (1,6);
\draw (6,2) -- (1,7);
\draw (6,3) -- (1,8);
\draw (6,4) -- (2,8);
\draw (6,5) -- (3,8);
\draw (6,6) -- (4,8);
\draw (6,7) -- (5,8);

\draw (2,5) circle (0.35);
\draw (2,7) circle (0.35);
\draw (4,2) circle (0.35);
\draw (4,7) circle (0.35);

\node[col1] () at (1,1) {};
\node[col1] () at (1,4) {};
\node[col1] () at (1,8) {};
\node[col2] () at (1,2) {};
\node[col2] () at (1,5) {};
\node[col2] () at (1,7) {};
\node[col3] () at (1,3) {};
\node[col3] () at (1,6) {};
\node[col4] () at (2,1) {};
\node[col4] () at (2,4) {};
\node[col4] () at (2,8) {};
\node[col5] () at (2,2) {};
\node[col9] () at (2,5) {};
\node[col7] () at (2,7) {};
\node[col6] () at (2,3) {};
\node[col6] () at (2,6) {};
\node[col7] () at (3,1) {};
\node[col7] () at (3,4) {};
\node[col7] () at (3,8) {};
\node[col8] () at (3,2) {};
\node[col8] () at (3,5) {};
\node[col8] () at (3,7) {};
\node[col9] () at (3,3) {};
\node[col9] () at (3,6) {};
\node[col1] () at (4,1) {};
\node[col1] () at (4,4) {};
\node[col1] () at (4,8) {};
\node[col7] () at (4,2) {};
\node[col2] () at (4,5) {};
\node[col9] () at (4,7) {};
\node[col3] () at (4,3) {};
\node[col3] () at (4,6) {};
\node[col4] () at (5,1) {};
\node[col4] () at (5,4) {};
\node[col4] () at (5,8) {};
\node[col5] () at (5,2) {};
\node[col5] () at (5,5) {};
\node[col5] () at (5,7) {};
\node[col6] () at (5,3) {};
\node[col6] () at (5,6) {};
\node[col1] () at (6,6) {};
\node[col1] () at (6,2) {};
\node[col2] () at (6,7) {};
\node[col2] () at (6,3) {};
\node[col2] () at (6,5) {};
\node[col3] () at (6,1) {};
\node[col3] () at (6,4) {};
\node[col3] () at (6,8) {};
\end{tikzpicture}
    \caption{The case $m \equiv 2 \pmod 3$, $n \equiv 1 \pmod 3$, illustrated by $T(5,7,5)$. The vertices of $S$ are circled.}
    \label{m2n1}
\end{figure}

Finally, we have the case $m \equiv 2 \pmod 3$, $n \equiv 2 \pmod 3$. Now rows $1$ and $n$ are bad. We recolour $c'((2,n)) = c'((m-1,1)) = 9$ and otherwise leave $c$ unchanged. Once again, $c'$ is proper, and the proof that it is odd is the same as in the previous case. This completes the proof that $T(m,n,t)$ has a nice colouring for $m \geq 3$.

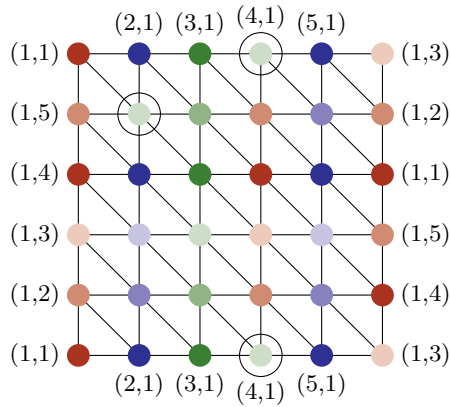
\begin{figure}[ht]
    \centering
\begin{tikzpicture}[scale=0.8]
\tikzset{col1/.style={fill, circle, inner sep=0pt, minimum size=3mm, color=Mahogany!100}}
\tikzset{col2/.style={fill, circle, inner sep=0pt, minimum size=3mm, color=Mahogany!50}}
\tikzset{col3/.style={fill, circle, inner sep=0pt, minimum size=3mm, color=Mahogany!20}}
\tikzset{col4/.style={fill, circle, inner sep=0pt, minimum size=3mm, color=Blue!100}}
\tikzset{col5/.style={fill, circle, inner sep=0pt, minimum size=3mm, color=Blue!50}}
\tikzset{col6/.style={fill, circle, inner sep=0pt, minimum size=3mm, color=Blue!20}}
\tikzset{col7/.style={fill, circle, inner sep=0pt, minimum size=3mm, color=OliveGreen!100}}
\tikzset{col8/.style={fill, circle, inner sep=0pt, minimum size=3mm, color=OliveGreen!50}}
\tikzset{col9/.style={fill, circle, inner sep=0pt, minimum size=3mm, color=OliveGreen!20}}

\foreach \y in {1,...,5} {%
    \node[left,xshift=-1mm] at (1,\y) {\footnotesize{(1,\y)}};
    }%
\foreach \y in {3,...,5} {%
    \node[right,xshift=1mm] at (6,\y-2) {\footnotesize{(1,\y)}};
    }%
\foreach \y in {1,2,3} {%
    \node[right,xshift=1mm] at (6,\y+3) {\footnotesize{(1,\y)}};
    }%
\node[left,xshift=-1mm] at (1,6) {\footnotesize{(1,1)}};
\foreach \x in {2,3,5} {%
    \node[below,yshift=-1mm] at (\x,1) {\footnotesize{(\x,1)}};
    \node[above,yshift=1mm] at (\x,6) {\footnotesize{(\x,1)}};
    }%
\node[below,yshift=-2mm] at (4,1) {\footnotesize{(4,1)}};
\node[above,yshift=2mm] at (4,6) {\footnotesize{(4,1)}};

\draw (1,1) grid (6,6);
\draw (2,1) -- (1,2);
\draw (3,1) -- (1,3);
\draw (4,1) -- (1,4);
\draw (5,1) -- (1,5);
\draw (6,1) -- (1,6);
\draw (6,2) -- (2,6);
\draw (6,3) -- (3,6);
\draw (6,4) -- (4,6);
\draw (6,5) -- (5,6);

\draw (2,5) circle (0.35);
\draw (4,1) circle (0.35);
\draw (4,6) circle (0.35);

\node[col1] at (1,1) {};
\node[col1] at (1,4) {};
\node[col1] at (1,6) {};
\node[col2] at (1,2) {};
\node[col2] at (1,5) {};
\node[col3] at (1,3) {};
\node[col4] at (2,1) {};
\node[col4] at (2,4) {};
\node[col4] at (2,6) {};
\node[col5] at (2,2) {};
\node[col9] at (2,5) {};
\node[col6] at (2,3) {};
\node[col7] at (3,1) {};
\node[col7] at (3,4) {};
\node[col7] at (3,6) {};
\node[col8] at (3,2) {};
\node[col8] at (3,5) {};
\node[col9] at (3,3) {};
\node[col9] at (4,1) {};
\node[col1] at (4,4) {};
\node[col9] at (4,6) {};
\node[col2] at (4,2) {};
\node[col2] at (4,5) {};
\node[col3] at (4,3) {};
\node[col4] at (5,1) {};
\node[col4] at (5,4) {};
\node[col4] at (5,6) {};
\node[col5] at (5,2) {};
\node[col5] at (5,5) {};
\node[col6] at (5,3) {};
\node[col1] at (6,4) {};
\node[col1] at (6,2) {};
\node[col2] at (6,5) {};
\node[col2] at (6,3) {};
\node[col3] at (6,1) {};
\node[col3] at (6,6) {};
\end{tikzpicture}
    \caption{The case $m \equiv 2 \pmod 3$, $n \equiv 2 \pmod 3$, illustrated by $T(5,5,3)$. The two recoloured vertices are circled; note that $(4,1)$ is shown twice.}
    \label{m2n2}
\end{figure}

\subsection*{$m = 2$}\hfill\\
We begin by colouring the vertices as in the $m \geq 3$ case: we use colours $\{1,2,3\}$ for column $1$ and $\{4,5,6\}$ for column 2, and $(i,j)$ is assigned the colour in the corresponding class that is equivalent to $j \pmod 3$, unless $n \equiv 1 \pmod 3$ in which case $(1,n)$ and $(2,n)$ receive colours $2$ and $5$ respectively. Let this colouring be $c$.

In the terminology of the $m \geq 3$ case, both columns are bad and there are at most two bad rows. If bad rows exist, then we can choose two good vertices $u$ and $w$ such that every bad vertex is adjacent to at least one of $u$ and $w$. We then recolour by using colours $7$ and $8$ at $u$ and $w$ respectively to create a new colouring $c'$. This colouring is trivially proper. It is also odd, since $c$ is odd at every vertex that is not adjacent to $u$ or $w$, and for every vertex that is adjacent to $u$ or $w$, either $7$ or $8$ appears exactly once in its neighbourhood in $c'$.

\begin{figure}[ht]
    \centering
\begin{tikzpicture}[scale=0.8]
\tikzset{col1/.style={fill, circle, inner sep=0pt, minimum size=3mm, color=Mahogany!100}}
\tikzset{col2/.style={fill, circle, inner sep=0pt, minimum size=3mm, color=Mahogany!50}}
\tikzset{col3/.style={fill, circle, inner sep=0pt, minimum size=3mm, color=Mahogany!20}}
\tikzset{col4/.style={fill, circle, inner sep=0pt, minimum size=3mm, color=Blue!100}}
\tikzset{col5/.style={fill, circle, inner sep=0pt, minimum size=3mm, color=Blue!50}}
\tikzset{col6/.style={fill, circle, inner sep=0pt, minimum size=3mm, color=Blue!20}}
\tikzset{col7/.style={fill, circle, inner sep=0pt, minimum size=3mm, color=OliveGreen!100}}
\tikzset{col8/.style={fill, circle, inner sep=0pt, minimum size=3mm, color=OliveGreen!50}}
\tikzset{col9/.style={fill, circle, inner sep=0pt, minimum size=3mm, color=OliveGreen!20}}

\draw (1,1) grid (3,6);
\draw (2,1) -- (1,2);
\draw (3,1) -- (1,3);
\draw (3,2) -- (1,4);
\draw (3,3) -- (1,5);
\draw (3,4) -- (1,6);
\draw (3,5) -- (2,6);

\draw (1,2) circle (0.35);
\draw (2,5) circle (0.35);
\draw (3,4) circle (0.35);

\foreach \y in {1,3,4,5} {%
    \node[left,xshift=-1mm] at (1,\y) {\footnotesize{(1,\y)}};
    }%
\node[left,xshift=-2.5mm] at (1,2) {\footnotesize{(1,2)}};
\node[left,xshift=-1mm] at (1,6) {\footnotesize{(1,1)}};
\node[below,yshift=-1mm] at (2,1) {\footnotesize{(2,1)}};
\node[above,yshift=1mm] at (2,6) {\footnotesize{(2,1)}};
\foreach \y in {1,3,4} {%
    \node[right,xshift=1mm] at (3,\y+2) {\footnotesize{(1,\y)}};
    }%
\node[right,xshift=2.5mm] at (3,4) {\footnotesize{(1,2)}};
\node[right,xshift=1mm] at (3,1) {\footnotesize{(1,4)}};
\node[right,xshift=1mm] at (3,2) {\footnotesize{(1,5)}};

\node[col1] at (1,1) {};
\node[col1] at (1,4) {};
\node[col1] at (1,6) {};
\node[col7] at (1,2) {};
\node[col2] at (1,5) {};
\node[col3] at (1,3) {};
\node[col4] at (2,1) {};
\node[col4] at (2,4) {};
\node[col4] at (2,6) {};
\node[col5] at (2,2) {};
\node[col8] at (2,5) {};
\node[col6] at (2,3) {};
\node[col1] at (3,3) {};
\node[col1] at (3,6) {};
\node[col1] at (3,1) {};
\node[col7] at (3,4) {};
\node[col2] at (3,2) {};
\node[col3] at (3,5) {};
\end{tikzpicture}
    \caption{$T(2,5,2)$ coloured as above. The two recoloured vertices are circled; note that $(1,2)$ is shown twice.}
    \label{m2case}
\end{figure}
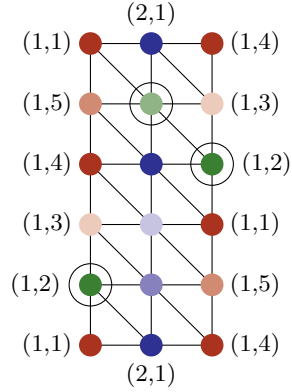

\subsection*{$m = 1$}\hfill\\
Let $H = T(1,n,t)$. We will refer to vertex $(1,j)$ as just $j$ for simplicity. Thus $j$ is adjacent to $j\pm 1$, $j \pm t$ and $j \pm (t+1)$, working modulo $n$. We can now see that $T(1,n,t)$ is isomorphic to $T(1,n,n-(t+1))$, and so $t < \f{n}{2}$ without loss of generality. Let $r = \left\lceil\f{n}{t}\right\rceil$, so that by the above we have $r \geq 3$. Note also that $t \geq 2$, otherwise the graph $T(1,n,t)$ would not be simple.

\begin{figure}[ht]
    \centering
\begin{tikzpicture}
\tikzset{enclosed/.style={draw, circle, inner sep=0pt, minimum size=1.5mm, fill=black}}
\pgfdeclarelayer{bg}
\pgfsetlayers{bg,main}

\foreach \x in {1,...,13} {%
    \begin{pgfonlayer}{bg}
    \draw[red] (360*\x/13+90:3cm) -- (360*4/13+360*\x/13+90:3cm);
    \draw[blue] (90+360*\x/13:3cm) -- (90+360*5/13+360*\x/13:3cm);
    \node at (90+360*\x/13:3.35cm) {\x};
    \end{pgfonlayer}{bg}
    \node[enclosed] (v\x) at (90+360*\x/13:3cm) {};
    }%
\draw (0,0) circle (3);
\end{tikzpicture}
    \caption{The graph $T(1,13,4)$. Edges $j \sim (j \pm 4)$ are shown in red and edges $j \sim (j \pm 5)$ in blue.}
    \label{m1example}
\end{figure}
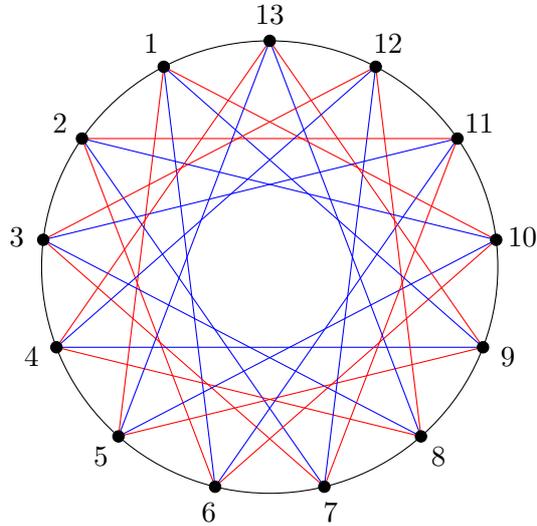

Now we partition the vertices of $T(1,n,t)$ into intervals: for $1 \leq k \leq \left\lfloor\f{n}{t}\right\rfloor$, let $I_k = \{j: (k-1)t+1 \leq j \leq kt\}$, and if $n$ is not a multiple of $t$, then $I_r = \{j: kt+1 \leq j \leq n\}$ consists of the remaining vertices. There are $r$ intervals in total.

As in the case $m \geq 3$, we split the colours into classes $C_1 = \{1,2,3\}$, $C_2 = \{4,5,6\}$ and $C_3 = \{7,8,9\}$. For each $C_i$ we now define the set of vertices $S_i$ at which $C_i$ will be used.

\begin{itemize}
    \item If $r \equiv 0 \pmod 3$, then $S_i = \cup_{k \equiv i \pmod 3}I_k$.
    \item If $r \equiv 1 \pmod 3$, then the $S_i$ are defined as for the $0 \pmod 3$ case except that $I_r$ is part of $S_2$ instead of $S_1$.
    \item If $r \equiv 2 \pmod 3$, then the $S_i$ are defined as for the $0 \pmod 3$ case except that $I_{r-1}$ is part of $S_2$ and $I_r$ is part of $S_3$.
\end{itemize}

\begin{figure}[ht]
    \centering
\begin{tikzpicture}
\draw (0,0) circle (1.5);
\draw (-4.5,0) circle (1.5);
\draw (4.5,0) circle (1.5);

\draw (-4.5,0) +(90:1.3) -- +(90:1.7) +(190:1.3) -- +(190:1.7) +(290:1.3) -- +(290:1.7) +(30:1.3) -- +(30:1.7);
\node (A) at (-4.5,0) {};
\path (A) ++(140:1.2) node () {$I_1$};
\path (A) ++(240:1.2) node () {$I_2$};
\path (A) ++(340:1.2) node () {$I_3$};
\path (A) ++(60:1.2) node () {$I_4$};
\path (A) ++(140:1.8) node () {\textcolor{Mahogany}{$S_1$}};
\path (A) ++(240:1.8) node () {\textcolor{Blue}{$S_2$}};
\path (A) ++(340:1.8) node () {\textcolor{OliveGreen}{$S_3$}};
\path (A) ++(60:1.8) node () {\textcolor{Blue}{$S_2$}};
\node at (-4.5,-2.2) {$r \equiv 1 \pmod 3$};

\draw (90:1.3) -- (90:1.7) (170:1.3) -- (170:1.7) (250:1.3) -- (250:1.7) (330:1.3) -- (330:1.7) (50:1.3) -- (50:1.7);
\node at (130:1.2) {$I_1$};
\node at (210:1.2) {$I_2$};
\node at (290:1.2) {$I_3$};
\node at (10:1.2) {$I_4$};
\node at (70:1.2) {$I_5$};
\node at (130:1.8) {\textcolor{Mahogany}{$S_1$}};
\node at (210:1.8) {\textcolor{Blue}{$S_2$}};
\node at (290:1.8) {\textcolor{OliveGreen}{$S_3$}};
\node at (10:1.8) {\textcolor{Blue}{$S_2$}};
\node at (70:1.8) {\textcolor{OliveGreen}{$S_3$}};
\node at (0,-2.2) {$r \equiv 2 \pmod 3$};

\draw (4.5,0) +(90:1.3) -- +(90:1.7) +(155:1.3) -- +(155:1.7) +(220:1.3) -- +(220:1.7) +(285:1.3) -- +(285:1.7) +(350:1.3) -- +(350:1.7) +(55:1.3) -- +(55:1.7);
\node (B) at (4.5,0) {};
\path (B) ++(122.5:1.2) node () {$I_1$};
\path (B) ++(187.5:1.2) node () {$I_2$};
\path (B) ++(252.5:1.2) node () {$I_3$};
\path (B) ++(317.5:1.2) node () {$I_4$};
\path (B) ++(22.5:1.2) node () {$I_5$};
\path (B) ++(72.5:1.2) node () {$I_6$};
\path (B) ++(122.5:1.8) node () {\textcolor{Mahogany}{$S_1$}};
\path (B) ++(187.5:1.8) node () {\textcolor{Blue}{$S_2$}};
\path (B) ++(252.5:1.8) node () {\textcolor{OliveGreen}{$S_3$}};
\path (B) ++(317.5:1.8) node () {\textcolor{Mahogany}{$S_1$}};
\path (B) ++(22.5:1.8) node () {\textcolor{Blue}{$S_2$}};
\path (B) ++(72.5:1.8) node () {\textcolor{OliveGreen}{$S_3$}};
\node at (4.5,-2.2) {$r \equiv 0 \pmod 3$};
\end{tikzpicture}
    \caption{The partition of $T(1,n,t)$ into intervals $I_i$ and subsets $S_i$. As in Figure \ref{m1example}, the graph is depicted as a circle.}
    \label{intervals}
\end{figure}
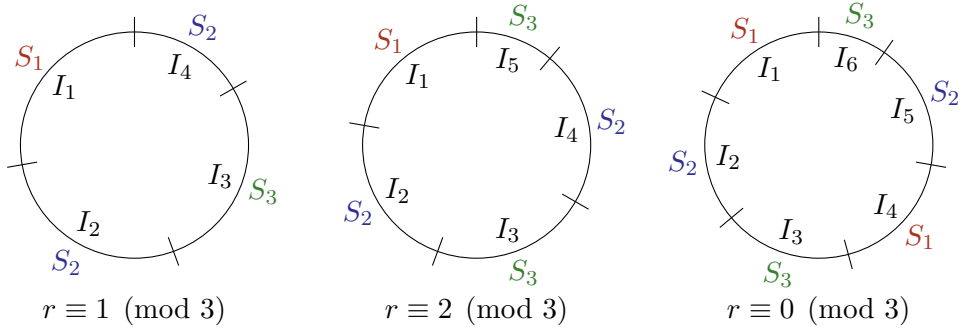

We now consider the induced graphs $H[S_i]$. The construction above ensures that the only edges between two vertices in the same interval are edges of the form $\{j,j+1\}$. In addition, the only edges between two distinct intervals in the same $S_i$ are edges of the form $\{kt,(k+1)t+1\}$ between the last vertex of one interval and the first vertex of another. This means that the induced graphs $H[S_i]$ are unions of disjoint paths, with the exception of the case $r = 4$, when $H[S_2]$ is a cycle. A path clearly has a proper odd $3$-colouring, so if $r \neq 4$ we can use the colours of $C_i$ to colour the vertices of each $S_i$ such that a proper odd colouring is induced on $H[S_i]$. This produces a nice colouring of $H$.

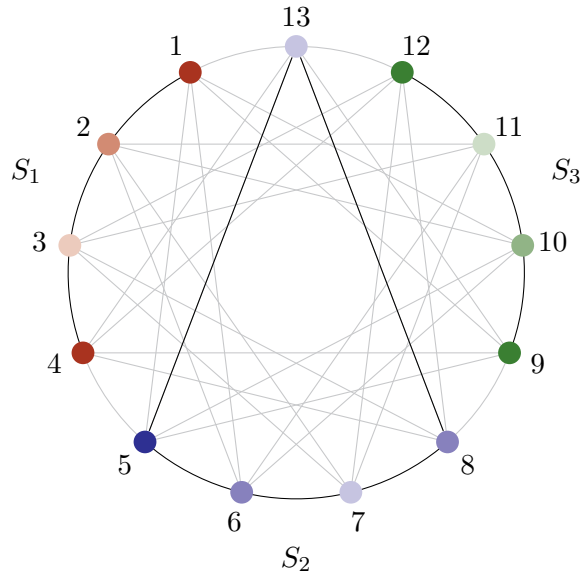
\begin{figure}[ht]
    \centering
\begin{tikzpicture}
\tikzset{enclosed/.style={draw, circle, inner sep=0pt, minimum size=1.5mm, fill=black}}
\tikzset{col1/.style={fill, circle, inner sep=0pt, minimum size=3mm, color=Mahogany!100}}
\tikzset{col2/.style={fill, circle, inner sep=0pt, minimum size=3mm, color=Mahogany!50}}
\tikzset{col3/.style={fill, circle, inner sep=0pt, minimum size=3mm, color=Mahogany!20}}
\tikzset{col4/.style={fill, circle, inner sep=0pt, minimum size=3mm, color=Blue!100}}
\tikzset{col5/.style={fill, circle, inner sep=0pt, minimum size=3mm, color=Blue!50}}
\tikzset{col6/.style={fill, circle, inner sep=0pt, minimum size=3mm, color=Blue!20}}
\tikzset{col7/.style={fill, circle, inner sep=0pt, minimum size=3mm, color=OliveGreen!100}}
\tikzset{col8/.style={fill, circle, inner sep=0pt, minimum size=3mm, color=OliveGreen!50}}
\tikzset{col9/.style={fill, circle, inner sep=0pt, minimum size=3mm, color=OliveGreen!20}}
\pgfdeclarelayer{bg}
\pgfsetlayers{bg,main}

\foreach \x in {1,...,13} {%
    \node at (90+360*\x/13:3.4) {\x};
    \begin{pgfonlayer}{bg}
    \draw[Gray!50] (360*\x/13+90:3cm) -- (360*4/13+360*\x/13+90:3cm);
    \draw[Gray!50] (90+360*\x/13:3cm) -- (90+360*5/13+360*\x/13:3cm);
    \end{pgfonlayer}{bg}
    }%
\draw (90+1*360/13:3) arc (90+1*360/13:90+4*360/13:3);
\draw (90+5*360/13:3) arc (90+5*360/13:90+8*360/13:3);
\draw (90-4*360/13:3) arc (90-4*360/13:90-1*360/13:3);
\draw[Gray!50] (90+4*360/13:3) arc (90+4*360/13:90+5*360/13:3);
\draw[Gray!50] (90+8*360/13:3) arc (90+8*360/13:90+9*360/13:3);
\draw[Gray!50] (90-1*360/13:3) arc (90-1*360/13:90+1*360/13:3);
\node[col1] at (90+1*360/13:3) {};
\node[col2] at (90+2*360/13:3) {};
\node[col3] at (90+3*360/13:3) {};
\node[col1] at (90+4*360/13:3) {};
\node[col4] (v5) at (90+5*360/13:3) {};
\node[col5] at (90+6*360/13:3) {};
\node[col6] at (90+7*360/13:3) {};
\node[col5] (v8) at (90+8*360/13:3) {};
\node[col7] at (90+9*360/13:3) {};
\node[col8] at (90-3*360/13:3) {};
\node[col9] at (90-2*360/13:3) {};
\node[col7] at (90-1*360/13:3) {};
\node[col6] (v13) at (90:3) {};
\node at (90+2.5*360/13:3.8) {$S_1$};
\node at (90+6.5*360/13:3.8) {$S_2$};
\node at (90-2.5*360/13:3.8) {$S_3$};

\draw (v5) -- (v13) -- (v8);

\end{tikzpicture}
    \caption{$T(1,13,4)$ coloured as below. Edges within the same $S_i$ are shown in black and edges between different $S_i$ in grey.}
    \label{finalexample}
\end{figure}

We are left with the case $r = 4$, where $3t+1 \leq n \leq 4t$. We begin by colouring the vertices of $S_1$ and $S_3$ using $C_1$ and $C_3$ respectively as in the previous case. Recall that $S_2$ consists of the union of $I_2 = \{j: t+1 \leq j \leq 2t\}$ and $I_4 = \{j: 3t+1 \leq j \leq n\}$. We properly $3$-colour the cycle $H[S_2]$ and apply the resulting colouring in $H$, giving no regard to oddness for the time being.

For $t+2 \leq j \leq 2t$, vertex $j$ is adjacent to exactly two vertices in $I_1$, and these are adjacent to each other and thus receive different colours from $C_1$. The colouring is therefore odd at $j$ for $t+2 \leq j \leq 2t$. Since $t \geq 2$, the same argument can be applied to vertex $t+1$ and interval $I_3$, showing that the colouring is odd at $j = t+1$ and therefore at every vertex in $I_2$.

If $n \geq 3t+2$, we can repeat this argument to show that the colouring is odd at every vertex of $I_4$, and therefore we have found a nice colouring of $H$. If instead $n = 3t+1$ then the argument breaks down at vertex $3t+1$. However, we know that the cycle $H[S_2]$ can be properly $3$-coloured in such a way that there are at most two vertices at which the colouring is not odd: indeed we made heavy use of the necessary constructions in the case $m \geq 3$. Clearly we can ensure that vertex $3t+1$ is not one of these vertices. Thus the colouring of $H[S_2]$ is odd at $3t+1$, and therefore the resulting colouring of $H$ is also odd at $3t+1$. This completes the proof of Theorem \ref{mainthm}.

\section{Conclusion}

First, we note that the proof of Theorem \ref{mainthm} also works for the real projective plane. Since the projective plane has Euler characteristic $1$, Euler's formula tells us that any graph $G$ on the projective plane has $\delta(G) \leq 5$, and that the total charge over all vertices and faces of $G$ in the discharging process is at most $-6$. Thus the discharging method alone suffices, and there is no special case such as the one dealt with in Section \ref{mindeg6}. We did not use any other properties of the torus that do not also hold for the real projective plane: like the torus, the projective plane is locally homeomorphic to the Euclidean plane, and orientation did not matter anywhere in our discharging proof. We therefore have the following result.

\begin{thm}\label{projplane}
Let $G$ be a simple graph that embeds in the real projective plane. Then $G$ has a proper odd colouring with at most $9$ colours.
\end{thm}

For a surface $S$, let $\chi_o(S)$ denote the maximum value of $\chi_o(G)$ over all simple graphs that can be embedded in $S$. We have therefore shown that for the torus $T$, $$7 \leq \chi_o(T) \leq 9.$$
We believe that in fact $7$ colours will always suffice.

\begin{conj}\label{7conj}
Let $G$ be a toroidal graph. Then $\chi_o(G) \leq 7$.
\end{conj}

Improving the upper bound beyond $9$ will likely require a new approach. Note that the two proofs that $8$ colours suffice for planar graphs, by Petr and Portier \cite{pp22} and Fabrici et al.\ \cite{fabrici22}, both use the Four Colour Theorem, and the bound of $8$ arises specifically as twice the bound of $4$ for ordinary colourings. 

The same methods applied to the torus would only produce an upper bound of $14$ colours for odd colourings. Indeed, we obtain a bound of $\chi_o(S) \leq 2\chi(S)$ for a general surface $S$, where $\chi(S)$ is the maximum of $\chi(G)$ over all graphs that embed in $S$. Any improvement on this bound for general surfaces $S$ would be of interest. In particular, it is natural to ask the following:

\begin{question}
Does $\chi(S) = \chi_o(S)$ for every surface other than the plane?
\end{question}

\section{Acknowledgements}

The author would like to thank his PhD supervisor Professor B\'ela Bollob\'as for his helpful comments, and Jan Petr and Julien Portier for introducing him to the topic and for their encouragement.

\end{document}